\documentclass[english]{article}
\usepackage[T1]{fontenc}
\usepackage[latin9]{inputenc}
\usepackage{geometry}
\usepackage{color}
\usepackage{babel}
\usepackage{refstyle}
\usepackage{units}
\usepackage{amsthm}
\usepackage{amsmath}
\usepackage{amssymb}
\usepackage[all]{xy}
\usepackage[unicode=true,pdfusetitle,
 bookmarks=true,bookmarksnumbered=false,bookmarksopen=false,
 breaklinks=false,pdfborder={0 0 1},backref=false,colorlinks=true]
 {hyperref}
\usepackage{breakurl}

\makeatletter

\AtBeginDocument{\providecommand\subref[1]{\ref{sub:#1}}}
\AtBeginDocument{\providecommand\thmref[1]{\ref{thm:#1}}}
\AtBeginDocument{\providecommand\eqref[1]{\ref{eq:#1}}}
\AtBeginDocument{\providecommand\secref[1]{\ref{sec:#1}}}
\AtBeginDocument{\providecommand\lemref[1]{\ref{lem:#1}}}
\AtBeginDocument{\providecommand\corref[1]{\ref{cor:#1}}}
\AtBeginDocument{\providecommand\Defref[1]{\ref{Def:#1}}}
\AtBeginDocument{\providecommand\propref[1]{\ref{prop:#1}}}
\AtBeginDocument{\providecommand\remref[1]{\ref{rem:#1}}}
\AtBeginDocument{\providecommand\exref[1]{\ref{ex:#1}}}
\RS@ifundefined{subref}
  {\def\RSsubtxt{section~}\newref{sub}{name = \RSsubtxt}}
  {}
\RS@ifundefined{thmref}
  {\def\RSthmtxt{theorem~}\newref{thm}{name = \RSthmtxt}}
  {}
\RS@ifundefined{lemref}
  {\def\RSlemtxt{lemma~}\newref{lem}{name = \RSlemtxt}}
  {}

\theoremstyle{plain}
\newtheorem{thm}{\protect\theoremname}
\theoremstyle{definition}
\newtheorem{defn}[thm]{\protect\definitionname}
\theoremstyle{remark}
\newtheorem{rem}[thm]{\protect\remarkname}
\theoremstyle{definition}
\newtheorem{example}[thm]{\protect\examplename}
\theoremstyle{plain}
\newtheorem{cor}[thm]{\protect\corollaryname}
\theoremstyle{plain}
\newtheorem{lem}[thm]{\protect\lemmaname}
\ifx\proof\undefined
\newenvironment{proof}[1][\protect\proofname]{\par
\normalfont\topsep6\p@\@plus6\p@\relax
\trivlist
\itemindent\parindent
\item[\hskip\labelsep
\scshape
#1]\ignorespaces
}{%
\endtrivlist\@endpefalse
}
\providecommand{\proofname}{Proof}
\fi
\theoremstyle{plain}
\newtheorem{prop}[thm]{\protect\propositionname}

\@ifundefined{date}{}{\date{}}
\newref{ex}{name=Example~}
\newref{sub}{name=Subsection~}
\newref{cor}{name=Corollary~}
\newref{rem}{name=Remark~}
\newref{prop}{name=Proposition~}
\newref{lem}{name=Lemma~}
\newref{thm}{name=Theorem~}

\hypersetup{citecolor=blue}

\usepackage{multicol}

\makeatother

\providecommand{\corollaryname}{Corollary}
\providecommand{\definitionname}{Definition}
\providecommand{\examplename}{Example}
\providecommand{\lemmaname}{Lemma}
\providecommand{\propositionname}{Proposition}
\providecommand{\remarkname}{Remark}
\providecommand{\theoremname}{Theorem}

\begin{document}
\global\long\def\dotimes{\widehat{\otimes}}
 \global\long\def\ZZ{\mathbb{Z}}
 \global\long\def\FF{\mathbb{F}}
 \global\long\def\NN{\mathbb{N}}

\title{{\huge{On regular $G$-gradings}}}

\maketitle
\begin{center}
{\large{Eli Aljadeff}}{\huge{}}%
\footnote{The first author was partially supported by the ISRAEL SCIENCE FOUNDATION
(grant No. 1283/08 and grant No. 1017/12) and by the GLASBERG-KLEIN
RESEARCH FUND.%
}\textit{\small{}}\\
\textit{\small{Department of Mathematics, Technion-Israel Institute
of Technology, Haifa 32000, Israel}}\\
\textit{\small{E-mail address: }}\texttt{\textit{\small{aljadeff@tx.technion.ac.il}}}
\par\end{center}{\small \par}

\begin{center}
{\large{and}}\\
{\large{$\;$}}\\
{\large{Ofir David}}\textit{\small{}}\\
\textit{\small{Department of Mathematics, Technion-Israel Institute
of Technology, Haifa 32000, Israel}}\\
\textit{\small{E-mail address: }}\texttt{\textit{\small{ofirdav@tx.technion.ac.il}}}
\par\end{center}{\small \par}
\begin{abstract}
Let $A$ be an associative algebra over an algebraically closed field
$\FF$ of characteristic zero and let $G$ be a finite abelian group.
Regev and Seeman introduced the notion of a regular $G$-grading on
$A$, namely a grading $A=\bigoplus_{g\in G}A_{g}$ that satisfies
the following two conditions: $(1)$ for every integer $n\geq1$ and
every $n$-tuple $(g_{1},g_{2},\dots,g_{n})\in G^{n}$, there are
elements, $a_{i}\in A_{g_{i}}$, $i=1,\dots,n$, such that $\prod_{1}^{n}a_{i}\neq0$
$(2)$ for every $g,h\in G$ and for every $a_{g}\in A_{g},b_{h}\in A_{h}$,
we have $a_{g}b_{h}=\theta_{g,h}b_{h}a_{g}.$ Then later, Bahturin
and Regev conjectured that if the grading on $A$ is regular and minimal,
then the order of the group $G$ is an invariant of the algebra. In
this article we prove the conjecture by showing that $ord(G)$ coincides
with an invariant of $A$ which appears in PI theory, namely $exp(A)$
(the exponent of $A$). Moreover, we extend the whole theory to (finite)
nonabelian groups and show that the above result holds also in that
case.
\end{abstract}

\section{Introduction and statement of the main results}

Group gradings on associative algebras (as well as on Lie and Jordan
algebras) have been an active area of research in the last 15 years
or so. In this article we will consider group gradings on associative
algebras over an algebraically closed field $\FF$ of characteristic
zero. The fact that a given algebra admits additional structures,
namely graded by a group $G$, provides additional information which
may be used in the study of the algebra itself , e.g. in the study
of group rings, twisted group rings and crossed products algebras
in Brauer theory (indeed ``gradings'' on central simple algebras
is an indispensable tool in Brauer theory, as it provides the isomorphism
of $Br(k)$ with the second cohomology group $H^{2}(G_{k},\bar{k}^{\times})$.
Here $k$ is any field, $G_{k}$ is the absolute Galois group of $k$
and $\bar{k}^{\times}$ denotes the group of units of the separable
closure of $k$). In addition, and more relevant to the purpose of
this article, $G$-gradings play an important role in the theory of
polynomial identities. Indeed, if $A$ is a PI-algebra which is $G$-graded,
then one may consider the $T$-ideal of $G$-graded identities (see
\subref{Graded-polynomial-identities}), denoted by $Id_{G}(A)$,
and it turns out in general, that it is easier to describe $G$-graded
identities than the ordinary ones for the simple reason that the former
ones are required to vanish on $G$-graded evaluations, rather than
on arbitrary evaluations. Nevertheless, two algebras $A$ and $B$
which are $G$-graded PI-equivalent are PI-equivalent as well, that
is, $Id_{G}(A)=Id_{G}(B)\Rightarrow Id(A)=Id(B)$.

Recall that a $G$-grading on an algebra $A$ is a vector space decomposition

\[
A\cong\bigoplus_{g\in G}A_{g}
\]
such that $A_{g}A_{h}\subseteq A_{gh}$ for every $g,h\in G$.

A particular type of $G$-gradings which is of interest was introduced
by Regev and Seeman in \cite{regev_$z_2$-graded_2005}, namely regular
$G$-gradings where $G$ is a finite abelian group. Let us recall
the definition.
\begin{defn}[Regular Grading]
Let $A$ be an associative algebra over a field $\FF$ and let $G$
be a finite abelian group. Suppose $A$ is $G$-graded. We say that
the $G$-grading on $A$ is regular if there is a commutation function
$\theta:G\times G\rightarrow F^{\times}$ such that
\begin{enumerate}
\item For every integer $n\geq1$ and every $n$-tuple $(g_{1},g_{2},\dots,g_{n})\in G^{n}$,
there are elements $a_{i}\in A_{g_{i}}$, $i=1,\dots,n$, such that
$\prod_{1}^{n}a_{i}\neq0$. 
\item For every $g,h\in G$ and for every $a_{g}\in A_{g},b_{h}\in A_{h}$,
we have $a_{g}b_{h}=\theta_{g,h}b_{h}a_{g}.$
\end{enumerate}
\end{defn}
\begin{rem}
One of our main tasks in this article is to extend the definition
above to groups which are not necessarily abelian and prove the main
results in that general context. For clarity we will continue with
the exposition of the abelian case and towards the end of the introduction
we will discuss extensions to the nonabelian setting. It seems to
us that the extension to the nonabelian case is rather natural in
view of the abelian case. In those cases below where the statement
in the general case is identical to the abelian case, we will make
a note indicating it. As for the proofs (sections \ref{sec:preliminaries}
and \ref{sec:Main-Theorem}), in case the result holds for arbitrary
groups, we present the general setting only, possibly with some remarks
concerning the abelian case.
\end{rem}
In the first section of this article we present examples of regular
$G$-gradings on finite and infinite dimensional algebras. We also
explain how can one compose algebras with regular gradings. One of
the most important examples of a regular grading on an algebra is
the well known $\nicefrac{\ZZ}{2\ZZ}$-grading on the infinite dimensional
Grassmann algebra $E$.
\begin{example}
\label{ex:grassmann_example}Let $E$ be the Grassmann algebra, defined
as the free algebra $\FF\left\langle e_{i}\mid i\in\NN\right\rangle $
with noncommuting variables, modulo the relations $e_{i}^{2}=0$ and
$e_{i}e_{j}=-e_{j}e_{i}$ for $i\neq j$. We set $E_{0}$ to be the
span of the monomials with even number of variables, and $E_{1}$
the span of monomials with odd number of variables.\\
 It is easy to see that $E=E_{0}\oplus E_{1}$, and this is actually
a regular $\nicefrac{\ZZ}{2\ZZ}$-grading with commutation function
$\tau_{0,0}=\tau_{0,1}=\tau_{1,0}=1$ and $\tau_{1,1}=-1$.\end{example}
\begin{rem}
It is sometimes more convenient to use the multiplicative group $C_{2}=\left\{ \pm1\right\} $
for the grading on the Grassmann algebra. Hence we also write $E=E_{1}\oplus E_{-1}$.
\end{rem}
The Grassmann algebra with its $\nicefrac{\ZZ}{2\ZZ}$-grading has
remarkable properties which are fundamental in the theory of PI-algebras.
Indeed, if $B=B_{0}\oplus B_{1}$ is a $\nicefrac{\ZZ}{2\ZZ}$-graded
algebra, we let $E\dotimes B$ be the Grassmann $\nicefrac{\ZZ}{2\ZZ}$-envelope
of $B$. Recall that the algebra $E\dotimes B$ is $\nicefrac{\ZZ}{2\ZZ}$-graded
as well and its grading is determined by $(E\dotimes B)_{0}=E_{0}\otimes B_{0}$
and $\left(E\dotimes A\right)_{1}=E_{1}\otimes B_{1}.$

A key property of the ``envelope operation'' is the following equality
of $\nicefrac{\ZZ}{2\ZZ}$-graded $T$-ideals of identities (and hence,
also of the corresponding ungraded $T$-ideals of identities). 
\begin{equation}
Id_{\nicefrac{\ZZ}{2\ZZ}}(E\dotimes(E\dotimes B)))=Id_{\nicefrac{\ZZ}{2\ZZ}}(B).\label{eq:Grassmann_envelope}
\end{equation}
We refer to the ``envelope operation'' as being \textit{involutive}.
It is well known that applying this operation, one can extend the
solution of the Specht problem and proof of ``representability''
from affine to non-affine PI-algebras (see \cite{kemer_ideals_1991},
\cite{aljadeff_representability_2009}).

Interestingly, the property satisfied by the Grassmann algebra we
just mentioned follows from the fact that the $\nicefrac{\ZZ}{2\ZZ}$-grading
on $E$ is regular, and indeed in \thmref{Envelope} we show that
a similar property holds for arbitrary regular graded algebras. In
order to state the result precisely we introduce the notion of $G$-envelope
of two algebras $A$ and $B$ where $G$ is a finite abelian group. 
\begin{defn}[$G$-envelope]
 Let $A,B$ be two $G$-graded algebras. We denote by $A\dotimes B$
the $G$-graded algebra defined by $\left(A\dotimes B\right)_{g}=A_{g}\otimes B_{g}$. 
\end{defn}
The following result generalizes \eqref{Grassmann_envelope}. The
proof is presented in \secref{preliminaries}.
\begin{thm}
\label{thm:Envelope} Let $A$ be a regularly $G$-graded algebra
with commutation function $\theta$, and let $B,C$ be two $G$-graded
algebras. 
\begin{enumerate}
\item If $\theta\equiv1$ then $Id_{G}(A\dotimes B)=Id_{G}(B)$.
\item Let $\tilde{A}=\dotimes^{\left|G\right|-1}A$ be the envelope of $\left|G\right|-1$
copies of $A$, then $\tilde{A}$ is regularly $G$-graded and $Id_{G}(\tilde{A}\dotimes(A\dotimes B))=Id_{G}(B)$.
\item $Id_{G}(B)=Id_{G}(C)$ if and only if $Id_{G}(A\dotimes B)=Id_{G}(A\dotimes C)$.
\end{enumerate}
\end{thm}
Our main goal in this article is to investigate the general structure
of (minimal) regular gradings on associative algebras over an algebraically
closed field of characteristic zero and in particular to give a positive
answer to conjecture 2.5 posed by Bahturin and Regev in \cite{bahturin_graded_2009}.

It is easy to see that a given algebra $A$ may admit regular gradings
with nonisomorphic groups and even with groups of distinct orders.
Therefore, in order to put some restrictions on the possible regular
gradings on an algebra $A$, Bahturin and Regev introduced the notion
of regular gradings which are \textit{minimal}. A regular $G$-grading
on an algebra $A$ with commutation function $\theta$ is said to
be \textit{minimal} if for any $e\neq g\in G$ there is $g'\in G$
such that $\theta(g,g')\neq1$.

Given a regular $G$-grading on an algebra $A$ with commutation function
$\theta$, one may construct a minimal regular grading with a homomorphic
image $\bar{G}$ of $G$. To see this, let 
\[
H=\{h\in G\mid\theta(h,g)=1:\mbox{for all }g\in G\}.
\]
One checks easily that $\theta$ is a skew symmetric bicharacter and
hence $H$ is a subgroup of $G$. Consequently, the commutation function
$\theta$ on $G$ induces a commutation function $\tilde{\theta}$
on $\bar{G}=G/H$. Moreover, the induced regular $\bar{G}$-grading
on $A$ is minimal.

In this article we consider the problem of uniqueness of a minimal
regular $G$-grading on an algebra $A$ (assuming it exists). It is
not difficult to show that an algebra $A$ may admit nonisomorphic
minimal regular $G$-gradings. Furthermore, an algebra $A$ may admit
minimal regular gradings with nonisomorphic abelian groups. However,
it follows from our results (as conjectured by Bahturin and Regev)
that the order of the group is uniquely determined. In fact, the order
of any group which provides a minimal regular grading on an algebra
$A$ coincides with a numerical invariant of $A$ which arises in
PI-theory, namely the PI-exponent of the algebra $A$ (denoted by
$exp(A)$). In order to state the result precisely we need some terminology
which we recall now.

Given a regular $G$-grading on an $\FF$-algebra $A$ we consider
the corresponding commutation matrix $M^{A}$ defined by $\left(M^{A}\right)_{g,h}=\theta(g,h)$,
$g,h\in G$ (see \cite{bahturin_graded_2009}). The commutation matrix
encodes properties of $\theta$. For instance, a regular grading is
minimal if and only if there is only one row of ones in $M^{A}$ (resp.
with columns).

Next we recall the definition of $\exp(A)$. For any positive integer
$n$ we consider the $n!$-dimensional $\FF$-space $P_{n}$, spanned
by all monomials of degree $n$ on $n$ different variables $\{x_{1},\ldots,x_{n}\}$
and let

\[
c_{n}(A)=\dim_{F}(P_{n}/(P_{n}\cap Id(A)).
\]
This is the $n$-th coefficient of the codimension sequence of the
algebra $A$. It was shown by Giambruno and Zaicev (see \cite{giambruno_codimension_1998},
\cite{giambruno_exponential_1999}) that the limit 
\[
\lim_{n\rightarrow\infty}c_{n}(A)^{1/n}
\]
exists and is a nonnegative integer. The limit is denoted by $\exp(A)$.

We can now state the main result of the paper in case the gradings
on $A$ are given by abelian groups.
\begin{thm}
\label{thm:Main_Theorem} Let $A$ be an algebra over an algebraic
closed field $\FF$ of characteristic zero and suppose it admits minimal
regular gradings by finite abelian groups $G$ and $H$.

Then
\begin{enumerate}
\item $\left|G\right|=\left|H\right|$ and this invariant is equal to $exp(A)$.
In particular the algebra $A$ is PI.
\item The commutation matrices $M_{G}^{A}$ and $M_{H}^{A}$ are conjugate.
In particular $tr(M_{G}^{A})=tr(M_{H}^{A})$ and $\det(M_{G}^{A})=\det(M_{H}^{A})$.
\item In fact, $\det(M_{G}^{A})=\pm\left|G\right|^{\left|G\right|/2}$.
\end{enumerate}
\end{thm}
\begin{rem}
Some of the results stated in Theorem \ref{thm:Main_Theorem} were
conjectured in \cite{bahturin_graded_2009}. Specifically, as mentioned
above, Bahturin and Regev conjectured that the order of a group which
provides a minimal regular gradings on an algebra $A$ is uniquely
determined. Moreover, they conjectured that if $M_{G}^{A}$ and $M_{H}^{A}$
are the commutation matrices of two minimal regular gradings on $A$
with groups $G$ and $H$ respectively, then $\det(M_{G}^{A})=\det(M_{H}^{A})\neq0$. 
\end{rem}

\subsection{Not necessarily abelian groups}

Suppose now that $G$ is an arbitrary finite group and let $A$ be
a $G$-graded algebra. As above, $A$ is an associative algebra over
an algebraically closed field $\FF$ of characteristic zero. We denote
by $A_{g}$ the corresponding $g\in G$-homogeneous component.
\begin{defn}
We say that the $G$-grading on $A$ is \emph{regular} if the following
two conditions hold. 
\begin{enumerate}
\item (commutation) For any $n$-tuple $(g_{1},\dots,g_{n})\in G^{n}$ and
for any permutation $\sigma\in Sym(n)$ such that the products $g_{1}\cdots g_{n}$
and $g_{\sigma(1)}\cdots g_{\sigma(n)}$ yield the same element of
$G$, there is a a nonzero scalar $\theta_{((g_{1},\dots,g_{n}),\sigma)}\in F^{\times}$
such that for any $n$-tuple $(a_{g_{1}},\ldots,a_{g_{n}})\in A^{n}$
with $a_{g_{i}}\in A_{g_{i}}$, the following equality holds 
\[
a_{g_{1}}\cdots a_{g_{n}}=\theta_{((g_{1},\dots,g_{n}),\sigma)}a_{g_{\sigma(1)}}\cdots a_{g_{\sigma(n)}}.
\]

\item (regularity) For any $n$-tuple $(g_{1},\dots,g_{n})\in G^{n}$, there
exists an $n$-tuple $(a_{g_{1}},\ldots,a_{g_{n}})\in A^{n}$ with
$a_{g_{i}}\in A_{g_{i}}$ such that $a_{g_{1}}\cdots a_{g_{n}}\neq0$.
\end{enumerate}
\end{defn}
\begin{rem}
In the special case where the elements $g,g'\in G$ commute we write
$\theta_{g,g'}$ instead of $\theta_{((g,g'),(12))}$. In particular
we will often use the notation $\theta_{g,g}$. Note that if $G$
is abelian, then $\theta_{\left((g_{1},...,g_{n}),\sigma\right)}$
is determined by $\theta_{g_{i},g_{j}},\;1\leq i,j\leq n$.
\end{rem}
Typical examples of regularly graded algebras ($G$ arbitrary) are
the well known group algebras $\FF G$,  and more generally, any
twisted group algebra $\FF^{\alpha}G$ where $\alpha$ is a $2$-cocycle
on $G$ with values in $\FF^{\times}$. Indeed, this follows easily
from the fact that each homogeneous component is $1$-dimensional
and every nonzero homogeneous element is invertible.

Additional examples can be obtained as follows.
\begin{enumerate}
\item If $A$ is a regularly $G$-graded algebra then $E\otimes A$ has
a natural regular $\nicefrac{\ZZ}{2\ZZ}\times G$-grading where $E$
is the infinite dimensional Grassmann algebra.
\item Let $A$ be a regularly $G$-graded algebra and suppose the group
$G$ contains a subgroup $H$ of index $2$.  Then we may view $A$
as a $\nicefrac{\ZZ}{2\ZZ}\cong G/H$-graded algebra and we let $A=A_{0}\oplus A_{1}$
be the corresponding decomposition. Let $E(A)=\left(E_{0}\otimes A_{0}\right)\oplus\left(E_{1}\otimes A_{1}\right)$
be the Grassmann envelope of $A$ and consider the following $G$-grading
on it. For any $g\in H$, we put $E(A)_{g}=E_{0}\otimes A_{g}$ whereas
if $g\not\in H$ we put $E(A)_{g}=E_{1}\otimes A_{g}$. We claim the
grading is regular. Indeed, let $(g_{1},\ldots,g_{n})\in G^{n}$ and
let $\sigma\in Sym(n)$ be a permutation such that $g_{1}\cdots g_{n}=g_{\sigma(1)}\cdots g_{\sigma(n)}$.
Then for elements $z_{g_{1}},\ldots,z_{g_{n}}$where $z_{g_{i}}\in E(A)_{g_{i}}$,
we have 
\[
z_{g_{1}}\cdots z_{g_{n}}=\tau((g_{1}H,\ldots,g_{n}H),\sigma)\theta((g_{1},\ldots,g_{n}),\sigma)z_{g_{\sigma}(1)}\cdots z_{g_{\sigma}(n)}
\]
where $\tau$ is the commutation function of the infinite Grassmann
algebra. For future reference we denote the commutation function
which corresponds to the regular $G$-grading on $E(A)$ by $\tau\theta$. 
\end{enumerate}
Following the discussion in the abelian case we define now nondegenerate
gradings for arbitrary finite groups as the counterpart of minimal
gradings. Let $A$ be an associative algebra and suppose it has a
regular $G$-grading with commutation function $\theta$. We say that
the grading is nondegenerate if for every $g\neq e$ in $G$, there
is an element $g'\in C_{G}(g)$ (the centralizer of $g$ in $G$)
such that $\theta_{(g,g^{'})}\neq1$.
\begin{rem}
It turns out (see \lemref{construct_algebra}) that $g\mapsto\theta_{g,g}$
is a homomorphism from $G$ to $\left\{ \pm1\right\} $ and therefore
its kernel $H=\{g\in G:\theta_{g,g}=1\}$ is a subgroup of $G$ (of
index $\leq2$). In case $H=G$, there is a cohomology class $\left[\alpha\right]\in H^{2}(G,\FF^{\times})$
such that $\FF^{\alpha}G$ has commutation function $\theta$. Then,
the nondegeneracy of the $G$-grading on $A$ corresponds to $\alpha$
being a nondegenerate $2$-cocycle. Groups $G$ which admit nondegenerate
$2$-cocycle are called ``central type''. It is a rather difficult
problem to classify central type groups. It is known, using the classification
of finite simple groups(!), that any central type group must be solvable.
It seems to be an interesting problem to classify finite groups which
admit nondegenerate commutation functions (modulo the classification
of central type groups).
\end{rem}
Our main results in the general case are extensions of the results
appearing in \thmref{Main_Theorem}.
\begin{thm}
\thmref{Main_Theorem}(1) holds for arbitrary nondegenerate regular
gradings (i.e. $G$ is not necessarily abelian).\end{thm}
\begin{rem}
In case $G$ is abelian, the commutation function $\theta$ is a skew
symmetric bicharacter (see Definition \ref{Def:bicharacter}). In
this case, it is a well known theorem of Scheunert \cite{scheunert1979generalized}
that $\theta$ arises from a 2-cocycle (as mentioned in the previous
remark). The notion of a bicharacter was considerably generalized
to cocommutative Hopf algebras (and hence in particular to group algebras)
(see \cite{Bahturin2001246}). Furthermore, whenever the bicharacter
is skew symmetric, the theorem of Scheunert can be extended to that
case. However, it should be noted that already for group algebras
$\FF G$ where $G$ is a nonabelian group, the linear extension of
$\beta(g,h)=f(g,h)/f(h,g)$ for a 2-cocycle $f:G\times G\to\FF^{\times}$
is not in general a skew symmetric bicharacter on $\FF G$.

Our results should be viewed or interpreted as to overcome this problem
by considering commutation functions which satisfy certain natural
necessary conditions in case they arise from 2-cocycles on $G$, and
then \lemref{construct_twisted_group_algebra} provides a generalization
of Scheunert's theorem to that context, namely every such commutation
function on $G$ ideed arises from a 2-cocycle on $G$. 
\end{rem}

\subsubsection{\label{sub:Commutation-matrix-Introduction}Commutation matrix}

As for the commutation matrix and its characteristic values, we need
to fix some notation. Suppose $A$ has a regular $G$-grading and
let $\theta^{A}$ be the corresponding commutation function. Suppose
first that $\theta_{g,g}^{A}=1$ for every $g\in G$. In that case
we know that the commutation function corresponds to an element $[\alpha]\in H^{2}(G,\FF^{\times})$.
With this data we consider the corresponding twisted group algebra
$B=\FF^{\alpha}G$ which is regularly $G$-graded with commutation
function $\theta^{A}$. It is well known that the algebra $\FF^{\alpha}G$
is spanned over $\FF$ by a set of invertible homogeneous elements
$\{U_{g}\}_{g\in G}$ that satisfy $U_{g}U_{h}=\alpha(g,h)U_{gh}$
for every $g,h\in G$.

Let us construct the corresponding commutation matrix. For every pair
$(g,h)\in G^{2}$ we consider the element $U_{g}U_{h}U_{g}^{-1}U_{h}^{-1}\in\FF^{\alpha}G$.
The matrix $M_{G}^{A}$ is determined by $(M_{G}^{A})_{(g,h)}=U_{g}U_{h}U_{g}^{-1}U_{h}^{-1}$
for every $g,h\in G$ and we note that this element does not depend
on the choice of the basis $\left\{ U_{g}:g\in G\right\} $.

Next we consider the general case. Let $\psi:G\to\FF^{\times}$ be
the map determined by $\psi(g)=\theta_{g,g}$. The function $\psi$
will be shown to be a homomorphism with its image contained in $\left\{ \pm1\right\} $
and we set $H=\ker(\psi)=\{g\in G:\theta_{g,g}=1\}$. Applying the
construction above we may define a $G$-grading on $E(A)$ where the
$\nicefrac{\ZZ}{2\ZZ}$-grading on $A$ is defined by $A=A_{H}\oplus A_{G\backslash H}$.
The commutation function $\tau\cdot\theta^{A}$ of $E(A)$ satisfies
$(\tau\theta^{A})_{g,g}=1$ for every $g\in G$. As in the previous
case the function $\tau\theta^{A}$ corresponds to a cohomology class
$\left[\alpha\right]\in H^{2}(G,F^{\times})$ were $\alpha$ is a
representing $2$-cocycle. We let $B=\FF^{\alpha}G$ be the corresponding
twisted group algebra with commutation function $\theta^{B}=\tau\theta^{A}$
and for every $g,h\in G$ we consider the element $U_{g}U_{h}U_{g}^{-1}U_{h}^{-1}\in\FF^{\alpha}G$.
The commutation matrix is defined by

\[
(M_{G})_{g,h}=\tau_{(\psi(g),\psi(h))}U_{g}U_{h}U_{g}^{-1}U_{h}^{-1}.
\]
We will usually write $\tau_{g,h}$ instead of $\tau_{\psi(g),\psi(h)}$.
Note that if $\psi\equiv1$, then $H=G$ and $\tau\mid_{G}\equiv1$,
so we have that $\theta^{B}=\theta^{A}$ as in the first case.
\begin{thm}
Let $A$ be an associative algebra over an algebraically closed field
of characteristic zero. Suppose $A$ admits a nondegenerate regular
$G$-grading and let $\theta$ be the corresponding commutation function.
Let $M_{G}$ be the commutation matrix constructed above. Then $M_{G}^{2}=\left|G\right|\cdot Id$.
\end{thm}
As a consequence we extend \thmref{Main_Theorem}(2,3) for arbitrary
nondegenerate regular gradings (see \subref{The-commutation-matrix},
\corref{comm_matrix}). In case $\theta_{g,g}=1$ for every $g\in G$,
we have that the elements $U_{g}U_{h}U_{g}^{-1}U_{h}^{-1}\in F^{f}G\cong M_{r}(F)$
and so the commutation matrix may be viewed as a matrix in $M_{r^{3}}(F)$.
In that case we obtain the following corollary.
\begin{cor}
$\det(M_{G})=\pm r^{(r^{3})}$, where $|G|=r^{2}$.
\end{cor}

\section{\label{sec:preliminaries}Preliminaries, examples and some basic
results}

In the first part of this section we recall some general facts and
terminology on $G$-graded PI-theory which will be used in the proofs
of the main results (we refer the reader \cite{aljadeff_representability_2009}
for a detailed account on this topic). In the second part of this
section we present some additional examples of regular gradings on
finite and infinite dimensional algebras. Finally, we present properties
of regular gradings and prove \thmref{Envelope}.

\subsection{\label{sub:Graded-polynomial-identities}Graded polynomial identities}

Let $W$ be a $G$-graded PI-algebra over $\FF$ and $I=Id_{G}(W)$
be the ideal of $G$-graded identities of $W$. These are polynomials
in $\FF\langle X_{G}\rangle$, the free $G$-graded algebra over $\FF$
generated by $X_{G}$, that vanish upon any admissible evaluation
on $W$. Here $X_{G}=\bigcup_{g\in G}X_{g}$ and $X_{g}$ is a set
of countably many variables of degree $g$. An evaluation is admissible
if the variables from $X_{g}$ are replaced only by elements of $W_{g}$.
It is known that $I$ is a $G$-graded $T$-ideal, i.e. closed under
$G$-graded endomorphisms of $\FF\langle X_{G}\rangle$.

We recall from \cite{aljadeff_representability_2009} that the $T$-ideal
$I=Id_{G}(W)$ is generated by multilinear polynomials and so it does
not change when passing to $\bar{\FF}$, the algebraic closure of
$\FF$, in the sense that the ideal of identities of $W_{\bar{\FF}}$
over $\bar{\FF}$ is the span (over $\bar{\FF}$) of the $T$-ideal
of identities of $W$ over $\FF$.

\subsection{Additional examples of regular gradings}

We present here some more examples (in addition to the ones presented
in the introduction). The following example corresponds to the grading
determined by the symbol algebra $(1,1)_{n}$.
\begin{example}
\label{ex:matrix} Let $M_{n}(\FF)$ be the matrix algebra over the
field $\FF$, and let $G=\nicefrac{\ZZ}{n\ZZ}\times\nicefrac{\ZZ}{n\ZZ}$.
For $\zeta$ a primitive $n$-th root of $1$ we define {\footnotesize{
\[
X=diag(1,\zeta,...,\zeta^{n-1})=\left[\begin{array}{ccccc}
1 & 0 &  & \cdots & 0\\
0 & \zeta & 0 &  & \vdots\\
 & 0 & \zeta^{2} & \ddots\\
\vdots &  & \ddots & \ddots & 0\\
0 & \cdots &  & 0 & \zeta^{n-1}
\end{array}\right]\qquad Y=E_{n,1}+\sum_{1}^{n-1}E_{i,i+1}=\left[\begin{array}{ccccc}
0 & 1 & 0 & \cdots & 0\\
 & 0 & 1 & \ddots & \vdots\\
\vdots &  & \ddots & \ddots & 0\\
0 &  &  & 0 & 1\\
1 & 0 & \cdots & 0 & 0
\end{array}\right].
\]
}}Note that $\zeta XY=YX$. Furthermore, the set $\left\{ X^{i}Y^{j}\mid0\leq i,j\leq n-1\right\} $
is a basis of $M_{n}(\FF)$, and so we can define a $G$-grading on
$M_{n}(\FF)$ by $(M_{n}(\FF))_{(i,j)}=\FF X^{i}Y^{j}$. Let us check
the $G$-grading is regular. For any two basis elements we have that
\begin{eqnarray*}
(X^{i_{1}}Y^{j_{1}})(X^{i_{2}}Y^{j_{2}}) & = & \zeta^{i_{2}j_{1}}X^{i_{1}}X^{i_{2}}Y^{j_{1}}Y^{j_{2}}=\zeta^{i_{2}j_{1}}X^{i_{2}}X^{i_{1}}Y^{j_{2}}Y^{j_{1}}=\zeta^{i_{2}j_{1}-i_{1}j_{2}}\left(X^{i_{2}}Y^{j_{2}}\right)\left(X^{i_{1}}Y^{j_{1}}\right)\\
 & \Rightarrow & \theta_{(i_{1},j_{1})(i_{2},j_{2})}=\zeta^{i_{2}j_{1}-i_{1}j_{2}}
\end{eqnarray*}
and hence the second condition in the definition of a regular grading
is satisfied. The first condition in the definition follows at once
from the fact that the elements $X$ and $Y$ are invertible. Finally
we note that since $\zeta$ is a primitive $n$-th root of unity,
the regular grading is in fact minimal.
\end{example}

\begin{example}
For any $n\in\NN$ and $c\in\FF^{\times}$, we can define a regular
$\nicefrac{\ZZ}{n\ZZ}$-grading on $A=\nicefrac{\FF\left[x\right]}{\left\langle x^{n}-c\right\rangle }$
by setting $A_{k}=\FF\cdot x^{k}$. Clearly, the commutation function
here is given by $\theta_{h,g}=1$ for all $g,h\in\nicefrac{\ZZ}{n\ZZ}$.
\end{example}

\begin{example}
For any algebra $A$ we have the trivial $G=\left\{ e\right\} $-grading
by setting $A_{e}=A$. In this case the grading is regular if and
only if $A$ is abelian and nonnilpotent.
\end{example}

\begin{example}
\label{ex:dihedral} We present an algebra with a \textit{nondegenerate}
regular $G$-grading where $G$ is isomorphic to the dihedral group
of order $8$.

Consider the presentation $\langle x,y:x^{4}=y^{2}=e,yxy^{-1}=x^{3}\rangle$
of the group $G$.

It is well known that there is a (unique) nonsplit extension 
\[
\xymatrix{\alpha_{G}:1\ar[r] & \left\{ \pm1\right\} \ar[r] & Q_{16}\ar[r]^{\pi} & G\ar[r] & 1}
\]
where $Q_{16}=\left\langle u,v\;:\; u^{8}=v^{4}=e,\;,u^{4}=v^{2},\; vuv^{-1}=u^{3}\right\rangle $
is isomorphic to the quaternion group of order $16$. The map $\pi$
is determined by $\pi(u)=x$ and $\pi(v)=y$. Note that the extension
is nonsplit on any nontrivial subgroup of $G$, that is, if $\{e\}\neq H\leq G$,
then the restricted extension

\[
\xymatrix{\alpha_{H}:1\ar[r] & \left\{ \pm1\right\} \ar[r] & \pi^{-1}(H)\ar[r]^{\pi} & H\ar[r] & 1}
\]
is nonsplit. Consider the twisted group algebra $\FF^{\alpha_{G}}G$
where the values of the cocycle are viewed in $\FF^{\times}$. Clearly,
$\FF^{\alpha_{G}}G$ is $G/K=C_{2}$-graded where $K$ is the Klein
$4$-group $K=\{e,x^{2},y,x^{2}y\}$ and so we can consider the corresponding
Grassmann envelope $E(\FF^{\alpha_{G}}G)$. We show that $E(\FF^{\alpha_{G}}G)$
is regularly $G$-graded and moreover the grading is nondegenerate.
Clearly the natural $G$-grading on the twisted group algebra $\FF^{\alpha_{G}}G$
is regular and hence the corresponding $G$-grading on $E(\FF^{\alpha_{G}}G)$
is also regular. Let $\theta$ be the corresponding commutation function.
To see that the $G$-grading on $E(\FF^{\alpha_{G}}G)$ is nondegenerate,
note that since the cocycle $\alpha_{H}$ is nontrivial on every subgroup
$H\neq\{e\}$ of $G$, the group $\pi^{-1}(K)$ is isomorphic to the
quaternion group of order $8$ and hence the twisted group subalgebra
$\FF^{\alpha_{K}}K$ of $\FF^{\alpha_{G}}G$ is isomorphic to $M_{2}(\FF)$.
This shows that the nondegeneracy condition (see Definition \Defref{G-comm-function})
is satisfied by any nontrivial element of $K$. For elements $g$
in $G\setminus K$ we have that $\theta_{g,g}=-1$ and we are done.
We will return to this example at the end of the paper.
\end{example}

\subsubsection{The commutation function $\theta$ and the commutation matrix}

We now turn to study some properties of the commutation function $\theta$.
We start with some notation.

Let $G$ be a group and $\bar{g}=(g_{1},...,g_{n})\in G^{n}$.
\begin{itemize}
\item Denote by $Sym(\bar{g})$ the set $\left\{ \sigma\in Sym(n)\;\mid\; g_{1}\cdots g_{n}=g_{\sigma(1)}\cdots g_{\sigma(n)}\right\} $.
\item For any $\sigma\in Sym(n)$ we write $\bar{g}^{\sigma}=(g_{\sigma(1)},...,g_{\sigma(n)})$. 
\end{itemize}
The conditions in the following lemma correspond to the properties
of T-ideal, namely $(1)$-closed to multiplication $(2)$-closed to
substitution and $(3)$-closed to addition.
\begin{lem}
\label{lem:prop_comm_function}Let $G$ be a group and $A$ a regularly
$G$-graded algebra with commutation function $\theta$. Then $\theta$
satisfies the following conditions.
\begin{enumerate}
\item Let $\bar{g}=(g_{1},...,g_{n})\in G^{n}$, $1\leq i\leq j\leq n$
and $\sigma\in Sym(\bar{g})$ such that $\sigma(k)=k$ for all $k<i$
or $k>j$. Denote by $\tilde{\sigma}\in Sym(j-i+1)$ the restriction
of $\sigma$ to $\left\{ i,i+1,...,j\right\} $, then $\theta_{(\bar{g},\sigma)}=\theta_{\left((g_{i,}g_{i+1},...,g_{j}),\tilde{\sigma}\right)}$.
\item Let $ $$\bar{h}=(h_{1},...,h_{k})\in G^{k}$. Let $\bar{g}_{i}=(g_{i,1},...,g_{i,n_{i}})\in G^{n_{i}}$
such that $\prod_{j=1}^{n_{i}}g_{i,j}=h_{i}$ and set $\bar{g}=(\bar{g}_{1},...,\bar{g}_{k})\in G^{\sum_{1}^{k}n_{i}}$.
For each $\sigma\in Sym(\bar{h})$ let $\tilde{\sigma}\in Sym(n_{1}+\cdots+n_{k})$
be the permutation that moves the blocks of size $n_{1},...,n_{k}$
according to the permutation $\sigma$. Then $\theta_{(\bar{h},\sigma)}=\theta_{(\bar{g},\tilde{\sigma})}$.
\item For every $g_{1},...,g_{n}\in G$ and $\sigma,\tau\in Sym(n)$ such
that $g_{1}\cdots g_{n}=g_{\sigma(1)}\cdots g_{\sigma(n)}=g_{\sigma\tau(1)}\cdots g_{\sigma\tau(n)}$
we have 
\[
\theta_{\left((g_{1},...,g_{n}),\sigma\right)}\theta_{\left((g_{\sigma(1)},...,g_{\sigma(n)}),\tau\right)}=\theta_{\left(g_{1},...,g_{n},\sigma\tau\right)}.
\]

\end{enumerate}
\end{lem}
\begin{proof}

\begin{enumerate}
\item This is an immediate consequence of the associativity of the product
in $A$.
\item The result follows from the fact that $A_{g_{i,1}}\cdots A_{g_{i,n_{i}}}\subseteq A_{g_{i,1}\cdots g_{i,n_{i}}}=A_{h_{i}}$.
\item Let $A=\bigoplus_{g\in G}A_{g}$ and choose some $a_{i}\in A_{g_{i}}$
such that $\prod a_{i}\neq0$. Then 
\begin{eqnarray*}
a_{1}\cdots a_{n} & = & \theta_{\left((g_{1},...,g_{n}),\sigma\right)}a_{\sigma(1)}\cdots a_{\sigma(n)}=\theta_{\left((g_{1},...,g_{n}),\sigma\right)}\theta_{\left((g_{\sigma(1)},...,g_{\sigma(n)}),\tau\right)}a_{\sigma\tau(1)}\cdots a_{\sigma\tau(n)}\\
a_{1}\cdots a_{n} & = & \theta_{\left(g_{1},...,g_{n},\sigma\tau\right)}a_{\sigma\tau(1)}\cdots a_{\sigma\tau(n)}.
\end{eqnarray*}
Finally, since $a_{1}\cdots a_{n}\neq0$, we have $\theta_{\left((g_{1},...,g_{n}),\sigma\right)}\theta_{\left((g_{\sigma(1)},...,g_{\sigma(n)}),\tau\right)}=\theta_{\left(g_{1},...,g_{n},\sigma\tau\right)}$
as desired. 
\end{enumerate}
\end{proof}
Next, we define $G$-commutation functions. We remind the reader that
if $g,h\in G$ commute we may denote by $\theta_{g,h}$ the scalar
$\theta_{((g,h),(1,2))}$.
\begin{defn}
\label{Def:G-comm-function}Let $\theta$ be a function from the pairs
$\bar{g}=(g_{1},...,g_{n})\in G^{n},\;\sigma\in Sym(\bar{g})$ with
values in $\FF^{\times}$. We say that $\theta$ is a $G$-commutation
function if it satisfies conditions (1, 2, 3) from the last lemma.
The function $\theta$ is said to be nondegenerate if for any $e\neq g\in G$
there is some $h\in C_{G}(g)$ such that $\theta_{g,h}\neq1$.
\end{defn}
In \lemref{construct_algebra} below we show that each $G$-commutation
function is in fact the commutation function of some regularly $G$-graded
algebra. By the definition, we get that a regular grading is nondegenerate
if and only if the commutation function is nondegenerate.
\begin{lem}
Let $\theta$ be a $G$-commutation function. Then the following hold.
\begin{enumerate}
\item For every $\bar{g}\in G^{n}$ we have $\theta_{(\bar{g},e)}\theta_{(\bar{g},e)}=\theta_{(\bar{g},e)}$
and so $\theta_{(\bar{g},e)}=1$.
\item For every commuting pair $g,h\in G$, we have $\theta_{g,h}=\theta_{h,g}^{-1}$.
\item For any fixed $g\in G$, the functions $h\mapsto\theta_{g,h}$ and
$h\mapsto\theta_{h,g}$ are characters on $C_{G}(g)$.
\end{enumerate}
\end{lem}
\begin{proof}

\begin{enumerate}
\item Part $(1)$ follows from part $(3)$ of the last lemma where $\sigma=\tau=id\in Sym(n)$. 
\end{enumerate}

Let $\sigma=(1,2)\in Sym(2)$. 
\begin{enumerate}
\item [2.]If $g,h$ commute, then $\theta_{g,h}\theta_{h,g}=\theta_{((g,h),\sigma)}\theta_{((h,g),\sigma)}=\theta_{((g,h),\sigma^{2})}=\theta_{((g,h),id)}=~1$.
\item [3.]By the conditions in \lemref{prop_comm_function} we get that
if $g\in G$ and $h_{1},h_{2}\in C_{G}(g)$, then 
\[
\theta_{g,h_{1}h_{2}}=\theta_{((g,h_{1}h_{2}),\sigma)}=\theta_{((g,h_{1},h_{2}),(1,3,2))}=\theta_{((g,h_{1},h_{2}),(1,2))}\theta_{((h_{1},g,h_{2}),(2,3))}=\theta_{g,h_{1}}\theta_{g,h_{2}}.
\]
Similarly we have that $\theta_{h_{1}h_{2},g}=\theta_{h_{1},g}\theta_{h_{2},g}$.
\end{enumerate}
\end{proof}
\begin{rem}
\label{rem:e_blocks}Notice in particular that $\theta_{e,g}=\theta_{g,e}=1$
for all $g\in G$. Using conditions (1) and (2) in \lemref{prop_comm_function}
we see that if $\sigma\in Sym(n)$ is a permutation which moves rigidly
in $\bar{g}=(g_{1},...,g_{n})\in G^{n}$ a block $\left(g_{i},g_{i+1},...,g_{j}\right)$
with $g_{i}\cdot g_{i+1}\cdots g_{j}=e$, then $\theta_{(\bar{g},\sigma)}=1$.
\end{rem}
If $G$ is abelian, then the commutation function $\theta_{(\bar{g},\sigma)}$
is defined by its values on pairs $\theta_{g,h}$. In that case we
get that $C_{G}(g)=G$ for all $g\in G$ and the conditions in \lemref{prop_comm_function}
follow from those in the last lemma. We recall the definition of such
functions.
\begin{defn}[Bicharacter]
\label{Def:bicharacter}Let $\eta:G\times G\rightarrow\FF^{\times}$
be a map where $G$ is a group and $\FF^{\times}$ is the group of
units of the field $\FF$. We say that the map $\eta$ is a bicharacter
of $G$ if for any $g_{0},h_{0}\in G$ the maps $h\mapsto\eta(g_{0},h)$
and $g\mapsto\eta(g,h_{0})$ are characters (i.e group homomorphisms
$G\rightarrow\FF^{\times}$). A bicharacter of $G$ is called \textit{skew-symmetric}
if $\eta(g,h)=\eta(h,g)^{-1}$ for any $h,g\in G$. A bicharacter
is said to be \textit{nondegenerate} if for any $e\neq g\in G$ there
is an element $h\in G$ such that $\theta(g,h)\neq1$. \end{defn}
\begin{rem}
In general, if $\theta$ is a commutation function on a finite group
$G$, then for any commuting elements $g,h\in G$ we have $ord(\;\theta(g,h)\;)\mid\; gcd(\; ord(g),\; ord(h)\;)$,
so $\theta(g,h)$ is contained in the group of roots of unity of order
$\left|G\right|$ in $\FF^{\times}$. In fact, as it will be shown
below, this holds for any $\theta_{(\bar{g},\sigma)}$. Also, we have
that $\theta(g,g)=\theta(g,g)^{-1}$ so $\theta(g,g)\in\left\{ \pm1\right\} $
for every $g\in G$.
\end{rem}
We present now two lemmas which summarize properties of the commutation
function and the ``$G$-envelope operation''. The proof of the first
lemma follows directly from the definitions and is left to the reader.
\begin{lem}
\label{lem:Oper_on_reg_alg}Suppose that $A,B$ are $G,H$-regulary
graded algebra with commutation functions $\theta$ and $\eta$ respectively.
Then the following hold.
\begin{enumerate}
\item $A\otimes B$ is a regularly $G\times H$-graded algebra with $(A\otimes B)_{(g,h)}=A_{g}\otimes B_{h}$
and $\left(\theta\otimes\eta\right)_{\left(((g_{1},h_{1}),...,(g_{n},h_{n}),\sigma\right)}=\theta_{((g_{1},...,g_{n}),\sigma)}\theta_{((h_{1},...,h_{n}),\sigma)}$
for all $\sigma\in Sym(\bar{g})\cap Sym(\bar{h})$.
\item If $G=H$ and $\theta=\eta$, then the algebra $A\oplus B$ is regularly
$G$-graded where $(A\oplus B)_{g}=A_{g}\oplus B_{g}$. Furthermore,
the corresponding commutation function is $\theta$.\\
In particular $\bigoplus_{1}^{n}A$ is regularly $G$-graded for any
$n\in\NN$.
\item If $N\leq G$ is a subgroup, then $A_{N}=\bigoplus_{g\in N}A_{g}$
is a regularly $N$-graded algebra with commutation function $\theta\mid_{N}$
-- the restriction to tuples in $N$.
\item If $G=H$ then $A\dotimes B$ is a regularly $G$-graded algebra with
commutation function $(\theta\dotimes\eta)_{(\bar{g},\sigma)}=\theta_{(\bar{g},\sigma)}\eta_{(\bar{g},\sigma)}$.
\end{enumerate}

If the groups $G,H$ are abelian, then the commutation matrix which
corresponds to the cases considered in the lemma are calculated as
follows: (1) $M^{A\otimes B}=M^{A}\otimes M^{B}$, (2) $M^{A\oplus B}=M^{A}$,
(3) $M^{A_{N}}$ is the restriction of $M^{A}$ to the group $N$,
and (4) $M^{A\dotimes B}$ is the Schur product (entry wise multiplication)
of $M^{A}$ and $M^{B}$.

\end{lem}
In the nonabelian case we have a similar connection between the commutation
matrices, though the ring over which the matrices are defined may
differ. More details are presented in the end of \subref{The-commutation-matrix}.\\

Suppose $A$ is regularly $G$-graded and let $\theta$ be the corresponding
commutation function. Given a multilinear polynomial $f(x_{g_{1},1},...,x_{g_{n},n})=\sum_{\sigma\in Sym(\bar{g})}\lambda_{\sigma}\prod x_{g_{\sigma(i)},\sigma(i)}$,
we denote by $f^{\theta}$ the polynomial 
\[
f^{\theta}(x_{g_{1},1},...,x_{g_{n},n})={\displaystyle \sum_{\sigma\in Sym(\bar{g})}}\lambda_{\sigma}\theta(\bar{g},\sigma)^{-1}\prod_{i}x_{g_{\sigma(i)},\sigma(i)}.
\]

\begin{lem}
\label{lem:envelope} Let $A$ be a regularly $G$-graded algebra
with commutation function $\theta$ and let $B$ be any $G$-graded
algebra. Let $f(x_{g_{1},1},...,x_{g_{n},n})=\sum_{\sigma\in Sym(\bar{g})}\lambda_{\sigma}\prod x_{g_{\sigma(i)},\sigma(i)}$.
Then $f^{\theta}\in Id_{G}(B)$ if and only if $f\in Id_{G}(A\dotimes B)$.\end{lem}
\begin{proof}
By multilinearity of $f$ we only check that $f$ vanishes on a spanning
set. For any $a_{i}\in A_{g_{i}}$ and $b_{i}\in B_{g_{i}}$ we get
that 
\begin{eqnarray*}
f(a_{1}\otimes b_{1},...,a_{n}\otimes b_{n}) & = & {\displaystyle \sum_{\sigma\in Sym(\bar{g})}}\lambda_{\sigma}\prod(a_{\sigma(i)}\otimes b_{\sigma(i)})={\displaystyle \sum_{\sigma\in Sym(\bar{g})}}\lambda_{\sigma}\prod_{i}a_{\sigma(i)}\otimes\prod_{i}b_{\sigma(i)}\\
 & = & \prod_{i}a_{i}\otimes{\displaystyle \sum_{\sigma\in Sym(\bar{g})}}\lambda_{\sigma}\theta(\bar{g},\sigma)^{-1}\prod_{i}b_{\sigma(i)}=\prod a_{i}\otimes f^{\theta}(b_{1},...,b_{n}).
\end{eqnarray*}
If $f^{\theta}\in Id_{G}(B)$, then the last term is zero so $f\in Id_{G}(A\dotimes B)$.
On the other hand, if $f\in Id_{G}(A\dotimes B)$, then the first
term is always zero. Since the grading on $A$ is regular, we can
find $a_{i}$ such that $\prod a_{i}\neq0$, so $\prod a_{i}\otimes f^{\theta}(b_{1},...,b_{n})=0$
if and only if $f^{\theta}(b_{1},...,b_{n})=0$ and we get that $f\in Id_{G}(B)$. 
\end{proof}
Before we proceed with the proof of \thmref{Envelope} we recall that
for any $G$-graded algebra over a field of characteristic zero $\FF$,
the $T$-ideal of $G$-graded identities is generated by multilinear
polynomials which are strongly homogeneous, namely polynomials of
the form

\[
f(x_{g_{1},1},...,x_{g_{n},n})=\sum_{\sigma\in Sym(\bar{g})}\lambda_{\sigma}\prod x_{g_{\sigma(i)},\sigma(i)}.
\]

\begin{proof}
of \thmref{Envelope}
\begin{enumerate}
\item This is immediate since $f^{\theta}=f$ when $\theta\equiv1$.
\item $\tilde{A}\dotimes A$ is the product of $\left|G\right|$ copies
of $A$. This is a regularly $G$-graded algebra with commutation
function $\theta^{\left|G\right|}\equiv1$. We now use the associativity
of the envelope operation and part (1) to conclude that $Id_{G}(\tilde{A}\dotimes(A\dotimes B))=Id_{G}(B)$. 
\item This follows immediately from the previous lemma.
\end{enumerate}
\end{proof}

\section{\label{sec:Main-Theorem}Main Theorem}

Our main objective in this section is to prove \thmref{Main_Theorem}.
The first step is to translate the definition of ``regular grading''
into the language of graded polynomial identities.
\begin{lem}
\label{lem:In_PI_words}Let $A$ be an algebra over $\FF$, $G$ a
finite group and $A=\bigoplus_{g\in G}A_{g}$ a $G$-grading on $A$.
Then the grading is regular if and only if the following conditions
hold.
\begin{enumerate}
\item There are no monomials with distinct indeterminates in $Id_{G}(A)$.
\item There is a function $\theta$ from pairs $\left(\bar{g},\sigma\right)$,
where $\bar{g}\in G^{n}$ and $\sigma\in Sym(\bar{g})$, such that
$x_{g_{1},1}\cdots x_{g_{n},n}-\theta_{(\bar{g},\sigma)}x_{g_{\sigma(1)},\sigma(1)}\cdots x_{g_{\sigma(n)},\sigma(n)}\in Id_{G}(A)$
(binomial identity).
\end{enumerate}
\end{lem}
\begin{proof}
The proof is clear. Indeed, condition (1) (resp. (2)) of the lemma
is equivalent to the first (resp. second) condition in the definition
of a regular grading.
\end{proof}
As mentioned above, the conditions in \lemref{prop_comm_function}
correspond to the properties of the T-ideal $Id_{G}(A)$, where $(1),(2),(3)$
correspond to closure under multiplication, closure under endomorphisms
and closure under addition respectively. Here is the precise statement.
\begin{prop}
\label{prop:relative_free_theta} Let $\theta$ be a commutation function
on a finite group $G$. Let $\FF\left\langle X_{G}\right\rangle $
be the graded free algebra over $\FF$ on the set $X_{G}$, where
$X_{G}=\{x_{g,i}:g\in G,i\in\mathbb{N}\}$ is a set of noncommuting
variables. For $\bar{g}\in G^{n}$, $\sigma\in Sym(\bar{g})$ and
$\bar{i}\in\NN^{n}$ we write $s(\bar{g},\sigma,\bar{i})=x_{g_{1},i_{1}}\cdots x_{g_{n},i_{n}}-\theta(\bar{g},\sigma)x_{g_{\sigma(1)},i_{\sigma(1)}}\cdots x_{g_{\sigma(n)},i_{\sigma(n)}}$.
Finally, we let $I$ be the $\FF$-subspace spanned by 
\[
S=\left\{ s(\bar{g},\sigma,\bar{i})\;\mid\;\bar{g}=(g_{1},...,g_{n})\in G^{n},\;\sigma\in Sym(\bar{g}),\; i_{1},...,i_{n}\in\NN\right\} .
\]
Then the following hold.
\begin{enumerate}
\item The vector space $I$ is a $T$-ideal.
\item The $G$-grading on $\nicefrac{\FF\left\langle X_{G}\right\rangle }{I}$
is regular with commutation function $\theta$.
\end{enumerate}
In particular, any $G$-commutation function is a commutation function
of some regular algebra.\end{prop}
\begin{proof}
The proof is based on translating the conditions of \ref{lem:prop_comm_function}
into the language of $T$-ideals. We give here only an outline of
the proof and leave the details to the reader.
\begin{enumerate}
\item By definition $I$ is closed under addition. To see $I$ is closed
under the multiplication of arbitrary polynomials, it is sufficient
to show it is closed under multiplication by $x_{g,j}$ for any $g\in G$
and $j\in\NN$ which is exactly condition (1) in \lemref{prop_comm_function}.
\\
Next we show the ideal $I$ is closed under endomorphisms. Notice
that if $s\in S$ is multilinear and $\varphi\in End(\FF\left\langle X_{G}\right\rangle )$,
one can decompose $\varphi=\varphi_{1}\circ\varphi_{2}$ such that
$\varphi_{2}$ sends each $x_{g_{j},i_{j}}$ to a sum of multilinear
monomials, all disjoint from each other, and $\varphi_{1}$ sends
each $x_{g,j}$ to some $x_{g',j'}$. It now follows from condition
(2) in \lemref{prop_comm_function} that $\varphi_{2}(s)=\sum s_{l}$
for some $s_{l}\in S$ multilinear, and that $\varphi_{1}(S)\subseteq S$.
This completes the proof.
\item The algebra $\nicefrac{\FF\left\langle X_{G}\right\rangle }{I}$ has
a natural $G$-grading, and by its definition it satisfies condition
2 in the definition of a regular grading. Therefore, we only need
to show that it has no monomial identities with distinct indeterminates.
\\
Condition (3) in \lemref{prop_comm_function} translates into the
equation 
\[
s(\bar{g},\sigma,\bar{i})+\theta(\bar{g},\sigma)s(\bar{g}^{\sigma},\tau,\bar{i}^{\sigma})=s(\bar{g},\sigma\tau,\bar{i}),\qquad\bar{i}^{\sigma}=(i_{\sigma(1)},...,i_{\sigma(n)})
\]
for any $\bar{g}\in G^{n},\;\bar{i}\in\NN^{n},\;\sigma\in Sym(\bar{g})$
and $\tau\in Sym(\bar{g}^{\sigma})$. For each $\bar{g}=(g_{1},...,g_{n})\in G^{n}$
and $\bar{i}\in\NN^{n}$, we define 
\begin{eqnarray*}
S(\bar{g},\bar{i}) & = & \left\{ s(\bar{g}^{\sigma},\tau,\bar{i}^{\sigma})\;\mid\;\sigma\in Sym(\bar{g}),\;\tau\in Sym(\bar{g}^{\alpha})\right\} \\
V(\bar{g},\bar{i}) & = & span\left\{ S(\bar{g},\bar{i})\right\} =span\left\{ s(\bar{g},\sigma,\bar{i})\;\mid\; e\neq\sigma\in Sym(\bar{g})\right\} .
\end{eqnarray*}
It is easy to see that if $I$ contains a monomial $x_{g_{1},i_{1}}\cdots x_{g_{n},i_{n}}$
with distinct indeterminates, then it must be in $V=V((g_{1},...,g_{n}),(i_{1},...,i_{n}))$.
The term $\prod x_{g_{\sigma(j)},i_{\sigma(j)}}$ for $e\neq\sigma\in Sym(\bar{g})$
appears only in $s(\bar{g},\sigma,\bar{i})$, so we see that $V$
does not contain monomials and we are done.
\end{enumerate}
\end{proof}
\begin{defn}
Let $\theta$ be a $G$-commutation function. The algebra $\nicefrac{\FF\left\langle X_{G}\right\rangle }{I}$
defined in the previous proposition is called the $\theta$-relatively
free algebra.
\end{defn}
Let $A$ be a $G$-graded algebra. Let $\pi:G\rightarrow\bar{G}$
be a surjective homomorphism and let $A={\displaystyle \bigoplus_{\bar{g}\in\bar{G}}}A_{\bar{g}}$
be the induced grading on $A$ by $\bar{G}$ (that is $A_{\bar{g}}={\displaystyle \bigoplus_{\pi(g)=\bar{g}}}A_{g}$).
Clearly, for multilinear polynomial $f$ we have $f(x_{\bar{g}_{1},1},...,x_{\bar{g}_{n},n})\in Id_{\bar{G}}(A)$
if and only if $f(x_{g_{1},1},...,x_{g_{n},n})\in Id_{G}(A)$ for
every $g_{i}\in G$ with $\pi(g_{i})=\bar{g}_{i}$ and so, in the
particular case where $\pi:G\rightarrow\left\{ e\right\} $, we obtain
the aforementioned fact that algebras which are $G$-graded PI-equivalent,
are also PI-equivalent. This simple but important fact will enable
us to replace the algebra $A$ by a more tractable $G$-graded algebra
$B$ (satisfying the same $G$-graded identities as $A$) from which
it will be easier to deduce the invariance of the order of the group
which provides a nondegenerate regular grading on $A$.

For the rest of this section, unless stated otherwise, we assume that
$\FF$ is algebraically closed and $char(\FF)=0$.
\begin{lem}
\label{lem:Id_G->Id}Let $A,B$ be two regularly $G$-graded algebras
with commutation functions $\theta_{A}$ and $\theta_{B}$ respectively.
If $\theta_{A}=\theta_{B}$ then $Id_{G}(A)=Id_{G}(B)$. In particular,
$Id(A)=Id(B)$. \end{lem}
\begin{proof}
Clearly, it is sufficient to consider multilinear polynomials. \\
Let $\bar{g}=(g_{1},...,g_{n})\in G^{n}$. Applying binomial $G$-graded
identities of $A$ (see \lemref{In_PI_words}), a polynomial $f(x_{g_{1},1},...,x_{g_{n},n})={\displaystyle \sum_{\sigma\in Sym(\bar{g})}\lambda_{\sigma}\prod_{i}x_{g_{\sigma(i)},\sigma(i)}}$
is a $G$-graded identity of $A$ if and only if $f'(x_{g_{1},1},...,x_{g_{n},n})={\displaystyle \left(\sum_{\sigma\in Sym(\bar{g})}\lambda_{\sigma}\theta_{(\bar{g},\sigma)}^{-1}\right)\prod_{i}x_{g_{i},i}}$
is a $G$-graded identity as well. But since there are no monomials
with distinct variables in $Id_{G}(A)$, $f'\in Id_{G}(A)$ if and
only if ${\displaystyle \sum_{\sigma\in Sym(\bar{g})}}\lambda_{\sigma}\theta_{(\bar{g},\sigma)}^{-1}=0$.
Thus, the statement 
\[
f(x_{g_{1},1},...,x_{g_{n},n})\in Id_{G}(A)
\]
is equivalent to a condition on the commutation function $\theta_{A}$
and the result follows.
\end{proof}
As mentioned above, we wish to replace any regularly $G$-graded algebra
$A$ with commutation function $\theta_{A}$ by a better understood
regularly $G$-graded algebra $B$ with commutation function $\theta_{B}=\theta_{A}$.

We first deal with the case where $\theta_{g,g}=1$ for all $g\in G$
(we remind the reader that in general $\theta_{g,g}=\pm1$ for all
$g\in G$). Here the algebra $B$ will be isomorphic to a suitable
twisted group algebra $B=\FF^{\alpha}G$, where $\alpha$ is a $2$-cocycle
on $G$ with values in $\FF^{\times}$. Recall that $B=\FF^{\alpha}G$
is isomorphic to the group algebra $\FF G$ as an $\FF$-vector space
and if $\{U_{g}:g\in G\}$ is an $\FF$-basis of $\FF^{\alpha}G$,
then the multiplication is defined by the rule $U_{g}U_{h}=\alpha(g,h)U_{gh}$
for every $g,h\in G$. It is well known that up to a $G$-graded isomorphism,
the twisted group algebra $\FF^{\alpha}G$ depends only on the cohomology
class of $\bar{\alpha}\in H^{2}(G,\FF^{\times})$ and not on the representative
$\alpha$. In order to construct the $2$-cocycle $\alpha=\alpha_{\theta}$,
we show that the commutation function $\theta=\theta_{A}$ (with $\theta_{g,g}=1$,
for all $g\in G$) determines uniquely an element in $Hom(M(G),\FF^{\times})$,
where $M(G)$ denotes the Schur multiplier of the group $G$. Then
applying the Universal Coefficient Theorem, we obtain an element in
$H^{2}(G,\FF^{\times})$ which by abuse of notation we denote again
by $\theta$.

Here is the precise statement and its proof.
\begin{lem}
\label{lem:construct_twisted_group_algebra}Let $\theta$ be a $G$-commutation
function such that $\theta_{g,g}=1$ for all $g\in G$. Then there
is a $2$-cocycle $\alpha\in Z^{2}(G,\FF^{\times})$ such that the
commutation function of $B=\FF^{\alpha}G$ is $\theta$.\end{lem}
\begin{proof}
The next construction follows the one in \cite{aljadeff_graded_2010}
(Prop. 1).

Recall that from the Universal Coefficient Theorem we get that for
any group $G$ we have an exact sequence 
\[
\xymatrix{1\ar[r] & Ext^{1}(G_{ab},\FF^{\times})\ar[r] & H^{2}(G,\FF^{\times})\ar[r]^{\pi} & Hom(M(G),\FF^{\times})\ar[r] & 1}
,
\]
where $M(G)$ is the Schur multiplier of $G$. Note that since $\FF$
is assumed to be algebraically closed, we have that $Ext^{1}(G_{ab},\FF^{\times})=0$
and hence in that case, the map $\pi$ is an isomorphism. Thus, our
task is to find a suitable element in $Hom(M(G),\FF^{\times})$ and
then show that its inverse image in $H^{2}(G,\FF^{\times})$ satisfies
the required property.

To start with, we fix a presentation of $M(G)$ via the Hopf formula:
Let $F$ be the free group $F=\left\langle y_{g}\;\mid\; g\in G\right\rangle $
and define $\varphi:F\to G$ by $\varphi(y_{g})=g$. Setting $R=\ker(\varphi)$
we have the exact sequence 
\[
\xymatrix{1\ar[r] & R\ar[r] & F\ar[r]^{\varphi} & G\ar[r] & 1}
.
\]
The Schur multiplier is then isomorphic to $\nicefrac{R\cap\left[F,F\right]}{\left[F,R\right]}$.

Next we show how an element $\alpha\in Z^{2}(G,\FF^{\times})$ determines
a map $\pi(\left[\alpha\right])$ on $\nicefrac{R\cap\left[F,F\right]}{\left[F,R\right]}$.
Let $\bar{g}=(g_{1},...,g_{n})\in G^{n}$ and $\sigma\in Sym(\bar{g})$.
Then $y_{g_{1}}\cdots y_{g_{n}}y_{g_{\sigma(n)}}^{-1}\cdots y_{g_{\sigma(1)}}^{-1}\in R\cap\left[F,F\right]$
and hence by the Hopf formula it determines an element in $M(G)$.
On the other hand, from \cite{aljadeff_graded_2010} we know that
any element in $M(G)$ has a presentation $y_{g_{1}}\cdots y_{g_{n}}y_{g_{\sigma(n)}}^{-1}\cdots y_{g_{\sigma(1)}}^{-1}\left[F,R\right]$
for some $\bar{g}\in G^{n}$ and $\sigma\in Sym(\bar{g})$, and moreover
the map 
\[
\pi\left(\left[\alpha\right]\right)(y_{g_{1}}\cdots y_{g_{n}}y_{g_{\sigma(n)}}^{-1}\cdots y_{g_{\sigma(1)}}^{-1}\left[R,F\right])=\frac{\alpha(g_{1},...,g_{n})}{\alpha(g_{\sigma(1)},...,g_{\sigma(n)})}
\]
is a well defined homomorphism.

Our next step is to show that $\psi(y_{g_{1}}\cdots y_{g_{n}}y_{g_{\sigma(n)}}^{-1}\cdots y_{g_{\sigma(1)}}^{-1}\left[R,F\right])=\theta(\bar{g},\sigma)$
is a well defined homomorphism. This will complete the proof of the
lemma.

Let $A=\nicefrac{\FF\left\langle X_{G}\right\rangle }{I}$ be the
$\theta$-relatively free algebra defined in \propref{relative_free_theta}.
Then $A$ is regularly $G$-graded with commutation function $\theta$.
If we can find elements $a_{g}\in A_{g}$, $g\in G$, which are invertible,
then we can define a group homomorphism $\tilde{\psi}:F\to A^{\times}$
induced by the map $y_{g}\mapsto a_{g}$. Notice that the image of
any commutator in $\left[R,F\right]$ is mapped to $1$ (because $R$
is mapped to $A_{e}$ which is in the center) while $y_{g_{1}}\cdots y_{g_{n}}y_{g_{\sigma(n)}}^{-1}\cdots y_{g_{\sigma(1)}}^{-1}$
is mapped to $\theta(\bar{g},\sigma)1$. The induced map $\varphi:M(G)\to\FF^{\times}$
is the required map. In general $A$ might not have such invertible
elements, so we need to construct new elements.

Let $S=\left\{ x_{g,1}x_{g^{-1},2}\right\} _{g\in G}\subseteq\FF\left\langle X_{G}\right\rangle $.
Note that it is sufficient to show that the elements in $S$ represent
nonzero divisors in $A$ since in that case, the localized algebra
$A'=AS^{-1}$ will still be regularly $G$-graded with commutation
function $\theta$ and in addition each $x_{g,1}$ will be invertible
(notice that $x_{g^{-1},2}x_{g,1}=\theta(g^{-1},g)x_{g,1}x_{g^{-1},2}$
and $\theta(g^{-1},g)=\theta(g,g)^{-1}=1$, so $x_{g,1}$ is right
and left invertible).

Suppose that there is some $0\neq f\in A$ such that $x_{g,1}x_{g^{-1}2}\cdot f\equiv0$.
We can assume that $f$ is homogeneous (i.e. its monomials have the
same $G$-homogeneous degree), and by standard methods (since the
field is infinite) we can assume that every variable $x_{h,i}$ appears
with the same total degree in each monomial of $x_{g,1}x_{g^{-1}2}\cdot f$
and therefore this is true also in $f$. Finally, using the binomial
identities we can assume that $f$ is a monomial. Now, by assumption
$x_{g,1}x_{g^{-1}2}\cdot f$ cannot be a monomial with different variables,
so we need to show there is no general monomial identities (i.e. with
possibly repeated variables).

Let $a_{1}x_{g,i}a_{2}x_{g,i}a_{3}\cdots a_{n}x_{g,i}a_{n+1}\in Id_{G}(A)$
be a monomial identity where $x_{g,i}$ does not appear in the monomials
$a_{i}$. Then using linearization we get that the polynomial

\[
f=\sum_{\sigma\in S_{n}}a_{1}z_{g,\sigma(1)}a_{2}z_{g,\sigma(2)}a_{3}\cdots a_{n}z_{g,\sigma(n)}a_{n+1}.
\]
is an identity as well. We claim that the monomials in $f$ are equal
modulo identities of $A$. In order to see this suppose that $a,b,c$
are monomials and denote by $h$ the degree of $a$. Let $g$ be some
element in $G$. Then 
\[
\left(y_{(hg)^{-1}}ax_{g,1}\right)bx_{g,2}c\equiv b\left(y_{(hg)^{-1}}ax_{g,1}\right)x_{g,2}c\equiv b\left(y_{(hg)^{-1}}ax_{g,2}\right)x_{g,1}c\equiv\left(y_{(hg)^{-1}}ax_{g,2}\right)bx_{g,1}c
\]
where the middle equation is true since $\theta_{g,g}=1$ and the
first and third equalities are true because monomials of degree $e$
are in the center. We therefore have $ax_{g,1}bx_{g,2}c\equiv ax_{g,2}bx_{g,1}c$.
Applying this equivalence we have that 
\[
\sum_{\sigma\in S_{n}}a_{1}y_{g,\sigma(1)}a_{2}y_{g,\sigma(2)}a_{3}\cdots a_{n}y_{g,\sigma(n)}a_{n+1}\equiv n!a_{1}y_{g,1}a_{2}y_{g,2}a_{3}\cdots a_{n}y_{g,n}a_{n+1}.
\]
Finally, we see that if we repeat this process for every pair $g\in G,\; i\in\mathbb{N}$
such that $x_{g,i}$ has total degree greater than $1$ in our monomial
identity, we obtain a monomial identity with distinct variables -
contradiction.
\end{proof}
Suppose that $A$ has a nondegenerate regular grading with commutation
function $\theta$ such that $\theta(g,g)=1$ for all $g\in G$. Let
$B=\FF^{\alpha}G$ as constructed in the last lemma. Clearly, the
twisted group algebra $B$ is regularly $G$-graded and the commutation
function is $\theta$. Invoking \lemref{Id_G->Id} we have the following
Corollary.
\begin{cor}
$Id_{G}(B)=Id_{G}(A)$. Consequently $Id(B)=Id(A)$. In particular
$exp(B)=exp(A)$.
\end{cor}
Our goal is to extract the cardinality of $G$ from $Id(B)$.

By Maschke's theorem, we know that any twisted group algebra $B=\FF^{\alpha}G$
is a direct sum of matrix algebras. We wish to show that the commutation
function $\theta$ is nondegenerate if and only if $B$ is simple,
or equivalently $\dim(Z(B))=1$. It is easily seen that the center
$Z(B)$ is spanned by elements of the form ${\displaystyle \sum_{\sigma\in G}}\lambda_{\sigma}U_{\sigma g\sigma^{-1}}$
where $g\in G$ and $\lambda_{\sigma}\in\FF$. We call a conjugacy
class that contribute a nonzero central element a \emph{ray class.}
The determination of the ray classes and their corresponding central
elements is well known (for example see \cite{Nauwelaerts198849},
section 2). The next lemma gives the condition for a conjugacy classes
to be a ray class and \lemref{generalized_simple_condition} will
generalize this idea to the $\nicefrac{\ZZ}{2\ZZ}$-simple case.
\begin{lem}
\label{lem:Ray_class}Let $g\in G$ and choose some set of left coset
representatives $\left\{ t_{i}\right\} _{1}^{k}$ of $C_{G}(g)$ in
$G$. For any 2-cocycle $\alpha\in Z^{2}(G,\FF^{\times})$ the following
conditions are equivalent:
\begin{enumerate}
\item For every $h\in C_{G}(g)$ we have $U_{g}U_{h}=U_{h}U_{g}$ in $\FF^{\alpha}G$.
\item The element $a={\displaystyle \sum_{i=1}^{k}}U_{t_{i}}U_{g}U_{t_{i}}^{-1}$
is central in $\FF^{\alpha}G$.
\end{enumerate}
In addition, if there are $\lambda_{i}\in\FF,\; i=1,...,k$, not all
zero, such that $b={\displaystyle \sum_{i=1}^{k}}\lambda_{i}U_{t_{i}}U_{g}U_{t_{i}}^{-1}$
is central in $\FF^{\alpha}G$ then $\lambda_{i}=\lambda_{1}$ for
all $i$. In particular we get that $a=\frac{1}{\lambda_{1}}b$ is
central in $\FF^{\alpha}G$.\end{lem}
\begin{proof}
Suppose first that $(1)$ holds. Let $w\in G$. Then for every $i\in\left\{ 1,...,k\right\} $,
there are $\tau(i)\in\left\{ 1,...,k\right\} $, $h_{i}\in C_{G}(g)$
and $c_{i}\in\FF^{\times}$ such that $U_{w}U_{t_{i}}=c_{i}U_{t_{\tau(i)}}U_{h_{i}}$.
Note that $\tau=\tau_{w}$ is a permutation of $\left\{ 1,...,k\right\} $.
Then we have that 
\begin{align*}
U_{w}aU_{w}^{-1} & =\sum\left(U_{w}U_{t_{i}}\right)U_{g}\left(U_{w}U_{t_{i}}\right)^{-1}=\sum\left(c_{i}U_{t_{\tau(i)}}U_{h_{i}}\right)U_{g}\left(c_{i}U_{t_{\tau(i)}}U_{h_{i}}\right)^{-1}\\
 & =\sum U_{t_{\tau(i)}}\left(U_{h_{i}}U_{g}U_{h_{i}}^{-1}\right)U_{t_{\tau(i)}}^{-1}=\sum U_{t_{\tau(i)}}U_{g}U_{t_{\tau(i)}}^{-1}=a,
\end{align*}
and so $U_{w}a=aU_{w}$. Since the set $\{U_{w}:w\in G\}$, spans
$\FF^{\alpha}G$, we get that $a$ is central.

On the other hand if $a$ is central and $h\in C_{G}(g)$, then there
is some $c\in\FF^{\times}$ such that $U_{h}U_{g}U_{h}^{-1}=cU_{g}$.
Assume that with $t_{1}\in C_{G}(g)$, so there is $c'\in\FF^{\times}$
such that $U_{t_{1}}U_{g}U_{t_{1}}^{-1}=c'U_{g}$. Thus we have 
\begin{eqnarray*}
a & = & \sum_{i=1}^{k}U_{t_{i}}U_{g}U_{t_{i}}^{-1}=c'U_{g}+\sum_{i=2}^{k}U_{t_{i}}U_{g}U_{t_{i}}^{-1}\\
a & = & U_{h}aU_{h}^{-1}=c'U_{h}U_{g}U_{h}^{-1}+\sum_{i=2}^{k}\left(U_{h}U_{t_{i}}\right)U_{g}\left(U_{h}U_{t_{i}}\right)^{-1}=c\cdot c'U_{g}+\sum_{i=2}^{k}U_{ht_{i}}U_{g}U_{ht_{i}}^{-1},
\end{eqnarray*}
so we must have that $c=1$ and we get that $(2)\Rightarrow(1)$.

Assume that $b=\sum\lambda_{i}U_{t_{i}}U_{g}U_{t_{i}}^{-1}$ is central
for some $\lambda_{i}\in\FF$ not all zero. For any $j\in\left\{ 1,...,k\right\} $
we get that 
\[
b=U_{t_{j}}^{-1}bU_{t_{j}}=\lambda_{j}U_{g}+{\displaystyle \sum_{i\neq j}}\lambda_{i}\left(U_{t_{j}}^{-1}U_{t_{i}}\right)U_{g}\left(U_{t_{i}}^{-1}U_{t_{j}}\right),
\]
so we must have $\lambda_{j}=\lambda_{1}$. In particular $\lambda_{1}\neq0$,
so $b=\lambda_{1}a$.
\end{proof}
By the last lemma, each ray class contributes only one central element
element up to a scalar multiplication, which we call a \emph{ray element}.
In addition, ray elements from different ray classes are linearly
independent. Thus we get that $\dim(Z(B))$ is the number of ray classes.
\begin{lem}
\label{lem:minimal+twisted}Let $B=\FF^{\alpha}G$, where $G$ is
a finite group of order $n$. Then $B$ is simple if and only if $\alpha$
is nondegenerate. Furthermore, in that case we have $B\cong M_{\sqrt{n}}(\FF)$.\end{lem}
\begin{proof}
As we remarked before the last lemma, $B$ is simple if and only if
$\dim(Z(B))=1$. Since $Z(B)$ is spanned by ray elements then $B$
is simple if and only if there is only one ray class (which is $\left\{ e\right\} $).
By the previous lemma this holds if and only if the cocycle $\alpha$
is nondegenerate.

Finally, we note that if $B$ is simple, then $B\cong M_{k}(\FF)$
and $k^{2}=\dim(B)=\left|G\right|=n$.
\end{proof}
We can complete now the proof of \thmref{Main_Theorem} in case the
commutation function satisfies $\theta_{g,g}=1$ for every $g\in G$.
Indeed, in \cite{regev_codimensions_1984} Regev showed that $\exp(M_{k}(\FF))=\dim\left(M_{k}(\FF)\right)=k^{2}$
and since the exponent of an algebra depends only on its ideal of
identities, we have from \lemref{Id_G->Id} that if $A$ has a regular
$G$-grading such that $\theta_{g,g}=1$ for all $g\in G$, then $\left|G\right|$
is an invariant of $A$ (as an algebra and independent of the grading).
\begin{cor}
Let $A$ be an algebra over an algebraically closed field $\FF$ of
characteristic zero. If $G$ is a finite group such that $A$ has
a nondegenerate regular $G$-grading with $\theta_{g,g}=1$ for all
$g\in G$ then $\left|G\right|=\exp(A)$.
\end{cor}
We move on to the general case where $\theta_{g,g}$ can be $-1$.
Let $H=\left\{ g\in G\;\mid\;\theta_{g,g}=1\right\} $. We are to
show that $H$ is a subgroup of $G$ of index $1$ or $2$. Then,
if the index is $1$, we are in the previous case where $\theta_{g,g}=1$
for all $g\in G$ whereas in the second case, we will find a twisted
group algebra for the group $G$ such that its Grassmann envelope
will be PI-equivalent to $A$.

Let $E=E_{1}\oplus E_{-1}$ be the infinite dimensional Grassmann
algebra over the field $\FF$, where $E_{1}$ and $E_{-1}$ are the
even and odd components of $E$. As noted above, this grading on $E$
is a regular $C_{2}$-grading with commutation function $\tau:C_{2}\times C_{2}\rightarrow\FF^{\times}$
determined by $\tau(1,1)=\tau(1,-1)=\tau(-1,1)=1$ and $\tau(-1,-1)=-1$.
\begin{lem}
\label{lem:construct_algebra}Let $\theta$ be a $G$-commutation
function. Then there is a $2$-cocycle $\alpha\in Z^{2}(G,\FF^{\times})$
and a subgroup $H\leq G$ such that for $B=\FF^{\alpha}G$, $B_{1}={\displaystyle \bigoplus_{h\in H}}B_{h},\; B_{-1}={\displaystyle \bigoplus_{g\notin H}}B_{g}$,
the Grassmann envelope $\tilde{B}=\left(B_{1}\otimes E_{1}\right)\oplus\left(B_{-1}\otimes E_{-1}\right)$
is a regularly $G$-graded algebra with commutation function $\theta$. \end{lem}
\begin{proof}
Let $\psi:G\rightarrow\left\{ \pm1\right\} $ be the map $\psi(g)=\theta_{g,g}$.
We claim that $\psi$ is a homomorphism. To see this, let $A$ be
the $\theta$-relatively free algebra (see \propref{relative_free_theta}).
For $h,g\in G$, let $\theta_{gh,gh}$ be the (unique) scalar such
that $x_{g,1}y_{h,1}x_{g,2}y_{h,2}=\theta_{gh,gh}x_{g,2}y_{h,2}x_{g,1}y_{h,1}$
in $A$ (we use $y_{h,*}$ instead of $x_{h,*}$ for clarity). From
\remref{e_blocks} we see that monomials of total degree $e$ are
central and hence we have

\begin{eqnarray*}
\left(w_{g^{-1}}x_{g,1}\right)y_{h,1}x_{g,2}\left(y_{h,2}z_{h^{-1}}\right) & = & y_{h,1}\left(y_{h,2}z_{h^{-1}}\right)\left(w_{g^{-1}}x_{g,1}\right)x_{g,2}=\theta_{h,h}\theta_{g,g}\left(y_{h,2}y_{h,1}\right)z_{h^{-1}}w_{g^{-1}}\left(x_{g,2}x_{g,1}\right)\\
 & = & \theta_{h,h}\theta_{g,g}y_{h,2}\left(y_{h,1}z_{h^{-1}}\right)\left(w_{g^{-1}}x_{g,2}\right)x_{g,1}=\theta_{h,h}\theta_{g,g}\left(w_{g^{-1}}x_{g,2}\right)y_{h,2}x_{g,1}\left(y_{h,1}z_{h^{-1}}\right)\\
w_{g^{-1}}\left(x_{g,1}y_{h,1}\right)\left(x_{g,2}y_{h,2}\right)z_{h^{-1}} & = & \theta_{gh,gh}w_{g^{-1}}\left(x_{g,2}y_{h,2}\right)\left(x_{g,1}y_{h,1}\right)z_{h^{-1}}.
\end{eqnarray*}

It follows that $\theta_{gh,gh}=\theta_{g,g}\theta_{h,h}$, and hence
letting $H=\ker(\psi)$ we have either $H=G$ or $\left[G:H\right]=2$.
The case where $H=G$ is the case considered above so we can assume
that $H\neq G$. In this case, roughly speaking, we apply first the
Grassmann envelope operation to ``turn the $-1$'s (in the image
of $\theta_{g,g}$) into $+1$'s'', then use the previous case to
find some $B=\FF^{\alpha}G$, and finally apply the Grassmann envelope
operation once again in order to return to the original identities.

To start with, we consider the group $\tilde{G}=\left\{ (g,\psi(g))\mid g\in G\right\} \leq G\times\nicefrac{\ZZ}{2\ZZ}$
which is clearly isomorphic to $G$. Then we define on $G$  a new
commutation function $\eta$ by $\eta_{((g_{1},...,g_{n}),\sigma)}=\theta_{((g_{1},...,g_{n}),\sigma)}\tau_{((\psi(g_{1}),...,\psi(g_{n})),\sigma)}$
where $\tau$ is the commutation function of the Grassmann algebra.
The function $\eta$ satisfies $\eta_{g,g}=\theta_{g,g}\tau_{\psi(g),\psi(g)}=1$
by the definition of $\psi$. We may apply now the case considered
above (that is when $\theta_{g,g}=1$, for all $g\in G$) and obtain
a suitable twisted group algebra $B=\FF^{\alpha}G$, $\alpha\in Z^{2}(G,\FF^{\times})$,
with commutation function $\eta$. We now apply the Grassmann envelope
operation once again. Let $\tilde{B}$ be the $\tilde{G}$ graded
algebra $\tilde{B}=(B\otimes E)_{\tilde{G}}=(B_{H}\otimes E_{0})\oplus(B_{G\backslash H}\otimes E_{1})$,
which is the Grassmann envelope of $B=B_{H}\oplus B_{G\backslash H}$.
Then $\tilde{B}$ is a regularly $\tilde{G}\cong G$-graded algebra
(since it is a subalgebra of the regularly $G\times C_{2}$-graded
algebra $B\otimes E$). We claim the commutation function $\tilde{\eta}$
of $\tilde{B}$ equals $\theta$. Indeed, 
\[
\tilde{\eta}_{((g_{1},...,g_{n}),\sigma)}=\eta_{((g_{1},...,g_{n}),\sigma)}\tau_{((\psi\left(g_{1}\right),...,\psi\left(g_{n}\right)),\sigma)}=\theta_{((g_{1},...,g_{n}),\sigma)}
\]
and so $\tilde{B}$ is the required envelope.
\end{proof}
Let us pause for a moment and summarize what we have so far. By the
previous lemma we have constructed an algebra $\tilde{B}$ which has
a regular $G$-grading whose commutation function coincides with a
given commutation function $\theta$ and hence, if $\theta$ is the
commutation function of a regularly $G$-graded algebra $A$, we have
in fact constructed a regularly $G$-graded algebra $\tilde{B}$ with
the same commutation function. It follows from \lemref{Id_G->Id}
that $Id_{G}(A)=Id_{G}(\tilde{B})$, $Id(A)=Id(\tilde{B})$ and hence
$\exp(A)=\exp(\tilde{B})$. The main point for constructing the algebra
$\tilde{B}$ is that in case the grading is nondegenerate, it enables
us to show that $ord(G)=\exp(\tilde{B})$. For this we need to further
analyze the algebra $\tilde{B}$ (constructed in \lemref{construct_algebra}).

Note that the algebras $B$ and $\tilde{B}$ in the lemma above satisfy
$\theta_{g_{1},g_{2}}=\tau_{\psi(g_{1}),\psi(g_{2})}\tilde{\theta}_{g_{1},g_{2}}$.
In particular if $h\in H$ and $g\in C_{G}(h)$ then $\theta_{g,h}=\tilde{\theta}_{g,h}$.
Since the grading on $\tilde{B}$ is nondegenerate then for every
$h\in H$ there is some $g\in C_{G}(h)$ with $\theta_{h,g}\neq1$,
and from what we just said, this is also true for $B$.
\begin{lem}
\label{lem:generalized_simple_condition}Let $G$ be a finite group
and $H$ a subgroup of index 2. Let $B=\FF^{\alpha}G$ be a twisted
group algebra such that for every $e\neq h\in H$ there is some $g\in C_{G}(h)$
such that $U_{h}U_{g}\neq U_{g}U_{h}$. Then the induced $\ZZ_{2}\cong\nicefrac{G}{H}$
grading on $B$ is $\ZZ_{2}$-simple. \end{lem}
\begin{proof}
Suppose first that the twisted group algebra $\FF^{\beta}H$, where
$\beta=\alpha\mid_{H}$, is simple. Let $0\neq I$ be a $\ZZ_{2}$-graded
ideal of $B$ and denote $I_{0}=I\cap B_{H}$ and $I_{1}=I\cap B_{G\backslash H}$,
so $I=I_{0}\oplus I_{1}$. Observe that since $I_{0}$ is an ideal
in $\FF^{\alpha}H$, it is either $0$ or $\FF^{\alpha}H$. On the
other hand, taking any $U_{g}$ where $g\notin H$, we have $U_{g}\cdot I_{0}\subseteq I_{1}$
and since $U_{g}$ is invertible in $\FF^{\alpha}G$ we have equality.
It follows that $I_{0}=\FF^{\alpha}H$ for otherwise $I=0$. We now
have 
\[
\dim(I)=\dim(I_{0})+\dim(I_{1})=2\dim(I_{0})=2\left|H\right|=\left|G\right|=\dim(B)
\]
so we see that $I=B$. This proves that $B$ is $\ZZ_{2}$-simple
in that case.

If $B$ is simple, then it must also be $\ZZ_{2}$-simple, so assume
that neither $B$ nor $\FF^{\beta}H$ are simple, or equivalently
both $\alpha$ and $\beta$ are degenerate $2$-cocycles. This means
that there is $h_{0}\in H$ such that $U_{h_{0}}U_{h}=U_{h}U_{h_{0}}$
for all $h\in C_{H}(h_{0})$, and similarly there is $g_{0}\in G$
such that $U_{g}U_{g_{0}}=U_{g_{0}}U_{g}$ for all $g\in C_{G}(g_{0})$.
Note that by the assumption on $B$ we must have $g_{0}\notin H$.
Let $\left\{ t_{i}\right\} ,\left\{ s_{i}\right\} $ be left coset
representatives of $C_{H}(h_{0})$ and $C_{G}(g_{0})$ respectively.
By \lemref{Ray_class} we have $a=\sum U_{t_{i}}U_{h_{0}}U_{t_{i}}^{-1}\in Z(\FF^{\beta}H)$
and $b=\sum U_{s_{i}}U_{g_{0}}U_{s_{i}}^{-1}\in Z(\FF^{\alpha}G)$.
If $s_{i}\notin H$ then $s_{i}g_{0}\in H$ is a representative of
the same left coset of $C_{G}(g_{0})$ as $s_{i}$, so we may assume
that $s_{i}\in H$ for all $i$.

By the assumption on $B$, there is some $g_{1}\in C_{G}(h_{0})$
such that $U_{g_{1}}U_{h_{0}}U_{g_{1}}^{-1}=cU_{h_{0}}$ with $c\neq1$,
and in particular $g_{1}\notin H$ by the choice of $h_{0}$. It is
easily seen that $\left\{ g_{1}t_{i}g_{1}^{-1}\right\} $ is again
a set of left coset representatives of $C_{H}(h_{0})$ in $H$ ($H$
is normal in $G$ and $g_{1}\in C_{G}(h_{0})$). We now have that
\[
U_{g_{1}}aU_{g_{1}}^{-1}=\sum U_{g_{1}t_{i}g_{1}^{-1}}cU_{h_{0}}U_{g_{1}t_{i}g_{1}}^{-1}=ca.
\]
Let $h\in H$ be such that $hg_{1}=g_{0}$. Then 
\begin{align*}
ab & =ba=\sum U_{s_{i}}U_{g_{0}}aU_{s_{i}}^{-1}=\alpha(h,g_{1})^{-1}\sum U_{s_{i}}U_{h}U_{g_{1}}aU_{s_{i}}^{-1}\\
 & =\alpha(h,g_{1})^{-1}ca\sum U_{s_{i}}U_{h}U_{g_{1}}U_{s_{i}}^{-1}=ca\sum U_{s_{i}}U_{g_{0}}U_{s_{i}}^{-1}=cab
\end{align*}
and we get a contradiction. Thus, we must have that either $\FF^{\alpha}G$
or $\FF^{\beta}H$ are simple. In both cases the algebra $\FF^{\alpha}G$
is $\ZZ_{2}\cong\nicefrac{G}{H}$-simple and the lemma is proved.\end{proof}
\begin{lem}
\label{lem:minimal-Z_2-simple}Let $G$ be a finite group and $H$
a subgroup of index 2. Let $B=\FF^{\alpha}G$ and let $\tilde{B}=\left(E_{1}\otimes B_{H}\right)\oplus\left(E_{-1}\otimes B_{G\backslash H}\right)$
be the Grassmann envelope of $B$. We denote by $\theta$ and $\tilde{\theta}$
the commutation functions of $B$ and $\tilde{B}$ respectively. If
the regular $G$-grading on $\tilde{B}$ is nondegenerate, then $B$
is a $\ZZ_{2}$-simple algebra. \end{lem}
\begin{proof}
By nondegeneracy of the grading, we have for any $e\neq h\in H$ an
element $g\in C_{G}(h)$ such that $\theta_{g,h}\neq1$ in $\tilde{B}$.
But the Grassmann envelope operation does not change this property
so it holds for the $G$-graded algebra $B$. Now use the previous
lemma.
\end{proof}
The fact that the algebra $B=\FF^{\alpha}G$ is finite dimensional
over $\FF$ ($\FF$ is algebraically closed of characteristic zero)
and $\ZZ_{2}$-simple almost determines the structure of $B$.
\begin{cor}
\label{cor:3_types_for_Z_2_simple}The algebra $B=\FF^{\alpha}G$
in the last lemma is $\ZZ_{2}$-isomorphic to one of the following
algebras.
\begin{enumerate}
\item $B=M_{n}(\FF)$ with the grading $B_{1}=B$ and $B_{-1}=0$.
\item $B=M_{n}(\FF)$ with the grading 
\begin{eqnarray*}
B_{1} & = & M_{(n,m)}^{1}=\left\{ \left(\begin{array}{cc}
D_{1} & 0\\
0 & D_{2}
\end{array}\right)\;\mid\; D_{1}\in\FF^{m\times m}\; D_{2}\in\FF^{(n-m)\times(n-m)}\right\} \\
B_{-1} & = & M_{(n,m)}^{-1}=\left\{ \left(\begin{array}{cc}
0 & D_{1}\\
D_{2} & 0
\end{array}\right)\;\mid\; D_{1}\in\FF^{m\times(n-m)}\; D_{2}\in\FF^{(n-m)\times m}\right\} 
\end{eqnarray*}

\item $B=M_{n}(\nicefrac{\FF[t]}{t^{2}=1})$ with the grading $B_{1}=M_{n}(\FF)$
and $B_{-1}=t\cdot M_{n}(\FF)$.
\end{enumerate}
\end{cor}
\begin{proof}
This is well known. See for instance Lemma 6 in \cite{kemer_finite_1987}.
\end{proof}
In our case, the algebra $B$ satisfies an additional condition, namely
$\dim(B_{1})=\dim(B_{2})=\left|H\right|$ so if $B$ is of the second
type above we must have $n=2m$.

We can now complete the proof of part $1$ of \thmref{Main_Theorem}.
\begin{cor}
Let $A$ be an algebra over an algebraically closed field $\FF$ of
characteristic 0. For every finite group $G$, if $A$ has a nondegenerate
regularly $G$-graded structure then $\left|G\right|=\exp(A)$. \end{cor}
\begin{proof}
We know that there is a simple $\ZZ_{2}$-graded algebra $B=\FF^{\alpha}G$
(where $B$ is one of the three types mentioned in the corollary above)
such that the algebra $\tilde{B}=E(B)$ satisfies $Id_{G}(\tilde{B})=Id_{G}(A)$.
In \cite{giambruno_codimension_2003} Giambruno and Zaicev computed
the exponent of Grassmann envelope of any $\ZZ_{2}$-simple algebras
and showed that $\exp(\tilde{B})=\exp(E(B))=\dim(B)=\left|G\right|$.
Because $Id_{G}(\tilde{B})=Id_{G}(A)$ we get that $\exp(A)=\exp(\tilde{B})=\left|G\right|$.
\end{proof}
We close this section with some additional corollaries of \lemref{construct_algebra}
and \lemref{minimal-Z_2-simple}. 

Let us denote the Grassmann envelope of the algebra $B$ in \corref{3_types_for_Z_2_simple}
(types $(2)$ and $(3)$ respectively) as follows:

\begin{eqnarray*}
M_{2m,m}(E) & = & \left[E_{1}\otimes M_{(2m,m)}^{1}(\FF)\right]\oplus\left[E_{-1}\otimes M_{(2m,m)}^{-1}(\FF)\right]\\
M_{n}(E) & = & \left[E_{1}\otimes M_{n}(\FF)\right]\oplus\left[E_{-1}\otimes t\cdot M_{n}(\FF)\right];\quad t^{2}=1.
\end{eqnarray*}

\begin{cor}
\label{cor:The-3-types}Suppose that $A$ has a nondegenerate regular
$G$-grading for some finite group $G$. Then one of the following
holds.
\begin{enumerate}
\item $Id(A)=Id(M_{n}(\FF))$ for some $n\in\NN$ and then $\exp(A)=n^{2}$.
\item $Id(A)=Id(M_{2m,m}(E))$ for some $m\in\NN$ and then $\exp(A)=(2m)^{2}$.
\item $Id(A)=Id(M_{n}(E))$ for some $n\in\NN$ and then $\exp(A)=2n^{2}$.
\end{enumerate}
\end{cor}
It is well known that the families considered in the corollary above
are mutually exclusive. Furthermore, different integers $n$ or $m$
yield algebras which are PI-nonequivalent. Indeed, algebras within
the same type are PI-nonequivalent as their exponent is different.
Next, any algebra of type one satisfies a Capelli polynomial whereas
any algebra of type $2$ or $3$ does not. Finally, the exponent of
any algebra of type $2$ is an exact square whereas this is not the
case for any algebra of type $3$. Thus if we let $\mathfrak{U}=\{M_{n}(\FF):n\in\NN\}\cup\{M_{2m,m}(E):m\in\NN\}\cup\{M_{n}(E):n\in\NN\}$
we have
\begin{cor}
Suppose that $A$ has a nondegenerate regular $G$-grading for some
finite group $G$. Then there is a unique algebra $C\in\mathfrak{U}$
such that $A$ and $C$ are PI-equivalent.
\end{cor}
From the results above we can now derive easily a consequence on the
commutation matrix $M^{A}$ for a regularly $G$-graded algebra $A$
with commutation function $\theta$.

The complete proof of \thmref{Main_Theorem} (parts $2$ and $3$)
is presented in the next section.

Recall that for nondegenerate regularly $G$-graded algebras of type
$1$ we have $\theta_{g,g}=1$ for all $g\in G$, whereas for type
$2$ and $3$ half of the entries on the diagonal of $M^{A}$ are
$1$'s and half are $-1$'s. From the definition of the commutation
matrix (see \subref{Commutation-matrix-Introduction}) we see that
$M_{g,g}^{A}=\theta_{g,g}U_{e}$. This clearly implies the following
corollary.
\begin{cor}
\label{cor:trace_invariant}Let $A$ be an $\FF$ algebra with a nondegenerate
regular $G$ grading and commutation matrix $M^{A}$. Then $tr(M^{A})$
is an invariant of $A$ and either $tr(M^{A})=0$ or $tr(M^{A})=\exp(A)U_{e}=\left|G\right|U_{e}$. 
\end{cor}

\subsection{\label{sub:The-commutation-matrix}The commutation matrix}

It is easy to exhibit algebras with nonisomorphic nondegenerate regular
$G$-gradings for some group $G$ as well as examples of algebras
with minimal regular gradings with nonisomorphic groups. For instance
the algebra $M_{4}(\FF)$ admits (precisely) two nonisomorphic minimal
gradings with the group $\nicefrac{\ZZ}{4\ZZ}\times\nicefrac{\ZZ}{4\ZZ}=\left\langle g,h\right\rangle $.
These gradings are determined by bicharacters $\theta_{1}$ and $\theta_{2}$,
where $\theta_{1}(g,h)=\zeta_{4}$ and $\theta_{2}(g,h)=\zeta_{4}^{3}$.
On the other hand the algebra $M_{2}(\FF)$ admits a (unique) nondegenerate
grading with the Klein $4$-group and hence the algebra $M_{4}(\FF)\cong M_{2}(\FF)\otimes M_{2}(\FF)$
admits a nondegenerate regular grading with the group $\left(\nicefrac{\ZZ}{2\ZZ}\right)^{4}$.

We therefore see that in general the entries of commutation matrices
which correspond to different nondegenerate regular gradings on an
algebra $A$ may be distinct. However, the last corollary shows that
the trace of the commutation matrices remains invariant.

Our goal in this section is to extend \corref{trace_invariant} and
show that any two such matrices corresponding to nondegenerate gradings
are conjugate (\thmref{Main_Theorem}).

We will follow the notation from \subref{Commutation-matrix-Introduction}.
In particular we have $B=\FF^{\alpha}G$, $H=\ker\left(g\mapsto\theta_{g,g}=\psi(g)\right)$
(a subgroup of $G$ of index $\leq2$) and $A$ is PI-equivalent to
the Grassmann envelope of $B$ with respect to the $\nicefrac{\ZZ}{2\ZZ}$-grading
$B=B_{H}\oplus B_{G\backslash H}$.

Before we consider nondegenerate gradings, let us analyze briefly
the degenerate case. If $G$ is abelian, the commutation matrix is
given by $M_{g,h}^{A}=\theta(g,h)U_{e}$. Hence, since the grading
is not minimal, there exists $g\neq e$ such that $\theta(g,h)=\theta(h,g)=1$
for all $h\in G$ and so $M^{A}$ is not invertible. The next proposition
shows that this is true in the nonabelian case as well.
\begin{prop}
Let $A$ be a regularly $G$-graded algebra with a degenerate grading.
Then $M^{A}$ is not invertible.\end{prop}
\begin{proof}
Let $B=\FF^{\alpha}G$ be the twisted group algebra which corresponds
to the $G$-graded algebra $A$ and denote by $\theta^{B}$ and $\theta^{A}$
the corresponding commutation functions. We note that for commuting
elements $g_{1},g_{2}\in G$ we have $\theta_{g_{1},g_{2}}^{B}=-\theta_{g_{1},g_{2}}^{A}$
if $\theta_{g_{1},g_{1}}^{A}=\theta_{g_{2},g_{2}}^{A}=-1$ and $\theta_{g_{1},g_{2}}^{B}=\theta_{g_{1},g_{2}}^{A}$
otherwise.

Since the field $\FF$ is algebraically closed of characteristic zero,
$B$ is a direct sum of matrix algebras over $\FF$. Fix a representation
$\rho:B\to M_{n}(\FF)$. The grading on $A$ is degenerate so there
is some $e\neq h\in G$ such that $\theta_{h,g}^{A}=1$ for all $g\in C_{G}(h)$
and in particular $\theta_{h,h}^{A}=1$. We thus have $\theta_{h,g}^{B}=1$
for all $g\in C_{G}(h)$. As a consequence, applying \lemref{Ray_class},
the element $z=\sum U_{t_{i}}U_{h}U_{t_{i}}^{-1}$, where $\left\{ t_{i}\right\} $
are left coset representatives of $C_{G}(h)$, is central in $B$.

Let $v\in{\displaystyle \prod_{g\in G}}B$ (a vector of size $G$
with entries in $B$) where $v_{g}=\lambda_{g}U_{g}$ for some $\lambda_{g}\in\FF$,
and consider 
\[
\left(M^{A}v\right)_{g}=\sum_{h\in G}M_{g,h}^{A}v_{h}=\sum_{h\in G}\tau_{g,h}U_{g}U_{h}U_{g}^{-1}U_{h}^{-1}\lambda_{h}U_{h}=U_{g}\left[\sum_{h\in G}\tau_{g,h}\lambda_{h}U_{h}\right]U_{g}^{-1}.
\]
Clearly, we may choose the $\lambda_{h}$'s such that $\sum_{h\in G}\tau_{g,h}\lambda_{h}U_{h}=\lambda_{1}U_{e}+\lambda_{2}z$
for all $g\in G$. This element is central so we have $\left(M^{A}v\right)_{g}=\lambda_{1}U_{e}+\lambda_{2}z$.
But the center of $M_{n}(\FF)$ is $\FF\cdot I$, so there is some
$c\in\FF$ such that $\rho(z)=c\cdot I$ and hence $\rho(\left(M^{A}v\right)_{g})=\left(\lambda_{1}+c\lambda_{2}\right)\cdot I$.
We see that if we choose $\lambda_{1},\lambda_{2}$ not both zero
such that $\lambda_{1}+c\lambda_{2}=0$ we have that $\rho(M^{A}v)=0$
for some $v\neq0$. Moreover, we note that the nonzero entries of
$v$ are invertible in $B$.

Let $\rho_{i}$ be the distinct representations of $B$ and let $e_{i}\in B$
be such that $\rho_{i}(e_{j})=\delta_{i,j}\cdot I$. For each $i$
let $v^{i}$ be a vector corresponding to $\rho_{i}$ as constructed
above and let $v=\sum v^{i}e_{i}\in{\displaystyle \prod_{g\in G}}B$.
Then $\rho_{i}(v_{g})=\rho_{i}(\sum v_{g}^{i}e_{j})=\rho_{i}(v_{g}^{i})$.
Furthermore, taking $g\in G$ such that $v_{g}^{i}\neq0$, we know
that $v_{g}^{i}$ is invertible and so $\rho_{i}(v_{g}^{i})\neq0$.
This implies that $v\neq0$. On the other hand we get for each $i$
\[
\rho_{i}(M^{A}v)=\sum_{j}\rho_{i}(M^{A}v^{i})\rho_{i}(e_{j})=\rho_{i}(M^{A}v^{i})=0
\]
and so $M^{A}v=0$. We conclude that $M^{A}$ is not invertible from
the left, and similar computations show that it is not invertible
from the right.
\end{proof}
Now we consider the case where the grading is nondegenerate.
\begin{prop}
Let $A$ be a nondegenerate regularly $G$-graded algebra, then $M^{A}\cdot M^{A}=\left|G\right|Id\cdot U_{e}.$\end{prop}
\begin{proof}
Recall that for any fixed $g\in G$, the function $\theta_{(\cdot,g)}^{A}:C_{G}(g)\to\FF^{\times}$
is a character, and since the grading is nondegenerate, this character
is nontrivial for $g\neq e$.

For fixed $a,c\in G$, set $N=C_{G}(a^{-1}c)$ and choose a set of
left coset representatives $\left\{ t_{i}\right\} _{i=1}^{[G:N]}\subset G$
of $N$ in $G$. Then
\begin{eqnarray*}
M_{a,c}^{2} & = & \sum_{b\in G}M_{a,b}M_{b,c}=\sum_{b\in G}\tau_{a,b}\tau_{b,c}U_{a}U_{b}U_{a}^{-1}U_{b}^{-1}U_{b}U_{c}U_{b}^{-1}U_{c}^{-1}\\
 & = & \sum_{b\in G}\tau_{ac^{-1},b}U_{a}U_{b}\left(U_{a}^{-1}U_{c}\right)U_{b}^{-1}U_{c}^{-1}\\
 & = & \sum_{i}U_{a}U_{t_{i}}\tau_{ac^{-1},t_{i}}\left[\sum_{h\in N}\tau_{ac^{-1},h}U_{h}\left(U_{a}^{-1}U_{c}\right)U_{h}^{-1}\right]U_{t_{i}}^{-1}U_{c}^{-1}\\
 & = & \sum_{i}U_{a}U_{t_{i}}\tau_{ac^{-1},t_{i}}\left[\sum_{h\in N}\tau_{ac^{-1},h}\theta_{h,a^{-1}c}^{B}\right]\left(U_{a}^{-1}U_{c}\right)U_{t_{i}}^{-1}U_{c}^{-1}.
\end{eqnarray*}
Notice that since $\psi:G\to\left\{ \pm1\right\} $, we have $\tau_{ac^{-1},h}=\tau_{h,ca^{-1}}=\tau_{h,a^{-1}c}$,
where $\tau_{g,h}=\tau_{\psi(g),\psi\left(h\right)}$ is the commutation
function of the Grassmann algebra with the $\ZZ_{2}$-grading. In
addition, the character $\theta^{A}(\cdot,a^{-1}c):N\to\FF^{\times}$
is nontrivial if and only if $a=c$ and so we get that 
\begin{align*}
\sum_{h\in H}\tau_{ac^{-1},h}\theta_{h,a^{-1}c}^{B} & =\sum_{h\in H}\tau_{h,a^{-1}c}\theta_{h,a^{-1}c}^{B}=\sum_{h\in H}\theta_{h,a^{-1}c}^{A}=\begin{cases}
0\qquad & a\neq c\\
\left|H\right| & a=c
\end{cases}.\\
\Rightarrow & M^{2}=\left|G\right|Id\cdot U_{e}.
\end{align*}
This completes the proof of the proposition.
\end{proof}
In the next discussion we use the notation of abelian groups, namely
$M^{A}\in M_{\left|G\right|}(\FF)$ with $M_{g,h}^{A}=\theta(g,h)$.
This can be generalized to the nonabelian (i.e. not necessarily abelian)
in the following way. We may view $B=\FF^{\alpha}G$ as a direct sum
of matrices, and then also $M_{\left|G\right|}(\FF^{\alpha}G)$ is
isomorphic to a direct sum of matrices. Alternatively, we may factor
through a representation $\rho:B\to M_{t}(\FF)$ of $B$ and then
extend it to $\rho:M_{\left|G\right|}(\FF^{\alpha}G)\to M_{\left|G\right|t}(\FF)$.
In any case, the matrix $M^{A}$ can be viewed as a matrix in $M_{k}(\FF)$
for some $k$ large enough.

It follows from the last proposition that the commutation matrix $M^{A}$
satisfies the polynomial $p(x)=x^{2}-n=(x-\sqrt{n})(x+\sqrt{n})$
where $n=\left|G\right|\neq0$. Hence, the corresponding minimal polynomial
is either $(x-\sqrt{n})$,$(x+\sqrt{n})$ or $p(x)$. In each case
the minimal polynomial has only simple roots and hence the matrix
$M^{A}$ is diagonalizable.

Let $\alpha^{+}$ and $\alpha^{-}$ denote the multiplicities of the
eigenvalues $\sqrt{n}$ and $-\sqrt{n}$ respectively. Then we have
$\alpha^{+}+\alpha^{-}=n$ and $\alpha^{+}-\alpha^{-}=\frac{tr(M^{A})}{\sqrt{n}}$.

In our case, $M^{A}$ has only $1$'s on the diagonal (the first type
of regular algebras), or half $1$'s and half $-1$'s (the second
and third type of regular algebras). Moreover, by \corref{trace_invariant}
we know that this depends only on the algebra $A$ and not on the
grading.

Thus, for algebras of the first type (in \corref{The-3-types}) we
have that $n=\exp(A)=\left|G\right|$ is a square and $tr(M^{A})=n$.
In that case, the equalities above take the form 
\[
\alpha^{+}+\alpha^{-}=n=m^{2},
\]
\[
\alpha^{+}-\alpha^{-}=\frac{tr(M^{A})}{\sqrt{n}}=\frac{n}{\sqrt{n}}=m
\]

and hence 
\[
\alpha^{+}=\binom{m+1}{2},\quad\alpha^{-}=\binom{m}{2}.
\]

For algebras of the second or third type (in \corref{The-3-types})
we have $n=\exp(A)$, which is either $2m^{2}$ or $(2m)^{2}$ for
some $m$, and $tr(M^{A})=0$. Then, here, the corresponding equalities
are

\[
\alpha^{+}+\alpha^{-}=n,
\]

\[
\alpha^{+}-\alpha^{-}=0
\]

and hence 
\[
\alpha^{+}=\alpha^{-}=\frac{n}{2}.
\]

\begin{cor}
\label{cor:comm_matrix}Suppose the algebra $A$ admits nondegenerate
regular gradings with groups $G$ and $H$ and let $M_{G}^{A}$ and
$M_{H}^{A}$ be the corresponding commutation matrices. Then the following
hold.
\begin{enumerate}
\item The matrices $M_{G}^{A}$ and $M_{H}^{A}$ are conjugate.
\item The characteristic and minimal polynomial of $M_{G}^{A}$ (and in
fact we may write ``of $M^{A}$'') are in $\ZZ\left[x\right]$.
\end{enumerate}
\end{cor}
\begin{proof}

\begin{enumerate}
\item By the proposition above we have $\left(M_{G}^{A}\right)^{2}=\left(M_{H}^{A}\right)^{2}=\exp(A)I$
and the trace is an invariant of $A$. Furthermore, the matrices $M_{G}^{A}$
and $M_{H}^{A}$ are both diagonalizable and have the same eigenvalues
(with multiplicities). In particular, $M_{G}^{A}$ and $M_{H}^{A}$
are conjugate.
\item If $n=\exp(A)$, then the eigenvalues of $M^{A}$ are $\pm\sqrt{n}$.
In particular, if $n$ is a square, then the minimal and characteristic
polynomials of $M^{A}$ are in $\ZZ\left[x\right]$.\\
In case $n=2m^{2}$ we have $\alpha^{+}=\alpha^{-}=\frac{n}{2}=m^{2}$
and so the characteristic polynomial is 
\begin{eqnarray*}
\prod_{1}^{m^{2}}(x-m\sqrt{2})\prod_{1}^{m^{2}}(x+m\sqrt{2}) & = & \prod_{1}^{m^{2}}\left[(x-m\sqrt{2})\left(x+m\sqrt{2}\right)\right]\\
 & = & \prod_{1}^{m^{2}}\left(x^{2}-2m^{2}\right)=\left(x^{2}-n\right)^{m^{2}}\in\ZZ\left[x\right]
\end{eqnarray*}
and the minimal polynomial is $(x-\sqrt{n})(x+\sqrt{n})=x^{2}-n$.
\end{enumerate}
\end{proof}
Finally, an easy computation of the free coefficient of the characteristic
polynomial in each one of the cases considered above yields that $\det(M_{G}^{A})=\pm\left|G\right|^{\left|G\right|/2}$.
This proves part $3$ of \thmref{Main_Theorem} and hence the entire
theorem is now proved.\\

As promised, we now compute the commutation matrix for the $G$-regular
algebras constructed in \lemref{Oper_on_reg_alg}, where $G$ is an
arbitrary finite group.

In case (2), the algebras $A,B$ and $A\oplus B$ have the same commutation
function $\theta$. Thus, the cocycles corresponding to $A,B,A\oplus B$
are isomorphic (up to a coboundary) and therefore the corresponding
twisted group algebras of $A,B$ and $A\oplus B$ are isomorphic.
With this identification of twisted group algebras we get that the
commutation matrices of $A,B$ and $A\oplus B$ are the same.

In case (3) we consider $A_{N}=\bigoplus_{g\in N}A_{g}$ for some
subgroup $N$ of $G$. Let $\alpha_{N}$ be the restriction of $\alpha$
to $N\times N$. Then $\alpha_{N}$ is the cocycle corresponding to
the algebra $A_{N}$ and there is a natural graded embedding $\FF^{\alpha_{N}}N\hookrightarrow\FF^{\alpha}G$.
Let $M'$ be the restriction of $M^{A}$ to the coordinates in $N\times N$,
then the entries of $M'$ are in $\FF^{\alpha_{N}}N$ and this submatrix
is actually $M^{A_{N}}$.

In cases (1) and (4) we have algebras $A,B$ with commutation functions
$\theta^{A},\theta^{B}$ and cocycles $\alpha,\beta$ which are defined
on groups $G$ and $H$ respectively. In case the groups $G$ and
$H$ are abelian, the matrix $M^{A\otimes B}$ is just $M^{A}\otimes M^{B}$.
For the general case, let $\alpha\otimes\beta\in Z^{2}\left(G\times H,\FF^{\times}\right)$
be the cocycle defined by $(\alpha\otimes\beta)((g_{1},h_{1}),(g_{2},h_{2}))=\alpha(g_{1},g_{2})\beta(h_{1},h_{2})$.
Clearly, $\alpha\otimes\beta$ represents the regular algebra $A\otimes B$
(with commutation function $\theta^{A}\otimes\theta^{B}$). Furthermore,
since $\FF^{\alpha}G\otimes\FF^{\beta}H\cong\FF^{\alpha\otimes\beta}(G\times H)$,
we can extend this product to a ``matrix tensor product''. In other
words, if $\varphi:\FF^{\alpha}G\otimes\FF^{\beta}H\to\FF^{\alpha\otimes\beta}(G\times H)$
is an isomorphism, then $M^{A\otimes B}$ is determined by $M_{(g_{1},h_{1}),(g_{2},h_{2})}^{A\otimes B}=\varphi(M_{g_{1},g_{2}}^{A}\otimes M_{h_{1},h_{2}}^{B})$.

In case (4), a similar computation shows that there is a isomorphism
$\psi:\FF^{\alpha}G\dotimes\FF^{\beta}G\to\FF^{\alpha\cdot\beta}G$
and then $M^{A\dotimes B}$ (which is defined over $\FF^{\alpha\beta}G$)
is determined by $M_{g,h}^{A\dotimes B}=\psi(M_{g,h}^{A}\otimes M_{g,h}^{B})$.
Finally, we note that there is a natural embedding $G\cong\tilde{G}=\left\{ (g,g)\mid g\in G\right\} \leq G\times G$.
Hence may view $A\dotimes B$ as $\left(A\otimes B\right)_{\tilde{G}}$
and with this identification $M^{A\dotimes B}$ is the restriction
of $M^{A\otimes B}$ to $\tilde{G}$.

\subsection{Nondegenerate skew-symmetric Bicharacters}

If the group $G$ is abelian then any $G$-commutation function $\theta$
is defined by the skew-symmetric bicharacter $\theta(g,h)=\theta_{g,h}$
for every $g,h\in G$ (the commutation of two elements).

Our goal in this section is to present a classification of the the
pairs $(G,\phi)$ where $G$ is a (finite) abelian group and $\phi$
is a nondegenerate skew-symmetric bicharacter defined on $G$. In
fact, this classification is known and can be found in \cite{alekseevich_commutation_1996}.
Nevertheless for the reader convenience and completeness of the article,
we recall the main results here.
\begin{defn}
Let $\theta_{1},\theta_{2}$ be two bicharacters on $G_{1},G_{2}$
respectively. We say that $\theta_{1},\theta_{2}$ are isomorphic
and write $\theta_{1}\cong\theta_{2}$ if there is an isomorphism
$\varphi:G_{1}\to G_{2}$ such that $\theta_{1}(g,h)=\theta_{2}(\varphi(g),\varphi(h))$.
\end{defn}

\begin{defn}
Let $\theta:G\times G\to\FF^{\times}$ be a skew-symmetric bicharacter.
We say that $\theta$ is reducible if there are groups $\left\{ e\right\} \neq H_{i}$
and bicharacters $\theta_{i}$ on $H_{i}$, for $i=1,2$, such that
$G\cong H_{1}\times H_{2}$ and $\theta\cong\theta_{1}\otimes\theta_{2}$.
\end{defn}
In what follows, we present three types of regularly graded algebras.
It turns out that the bicharacters which correspond to some special
cases of these gradings are irreducible and generate all possible
skew-symmetric bicharacters.
\begin{example}
The standard $\nicefrac{\ZZ}{2\ZZ}$-grading on the Grassmann algebra
(\exref{grassmann_example})
\end{example}

\begin{example}
The $\nicefrac{\ZZ}{n\ZZ}\times\nicefrac{\ZZ}{n\ZZ}$-grading on $M_{n}(\FF)$
defined in \exref{matrix}.
\end{example}

\begin{example}
Let $\zeta$ be a primitive root of unity of order $2n$, and define
\begin{eqnarray*}
X_{\zeta} & = & \left(\begin{array}{ccccc}
1 & 0 & \cdots & 0 & 0\\
0 & \zeta^{2} & 0 &  & 0\\
\vdots & 0 & \zeta^{4} & \ddots & \vdots\\
0 &  & \ddots & \ddots & 0\\
0 & 0 & \cdots & 0 & \zeta^{2n-2}
\end{array}\right),\quad Y=\left(\begin{array}{cccccc}
0 & 1 & 0 & \cdots & 0 & 0\\
0 & 0 & 1 & \ddots &  & 0\\
\vdots &  & \ddots & \ddots &  & \vdots\\
 &  &  &  & 1 & 0\\
0 &  &  &  & 0 & 1\\
1 & 0 & \cdots &  & 0 & 0
\end{array}\right)\in M_{n}(\FF)\\
U & = & \left(\begin{array}{cc}
X_{\zeta} & 0\\
0 & \zeta X_{\zeta}
\end{array}\right),\; V=\left(\begin{array}{cc}
0 & I\\
Y & 0
\end{array}\right)\in M_{2n}(\FF).
\end{eqnarray*}
Then 
\begin{eqnarray*}
UV & = & \left(\begin{array}{cc}
X_{\zeta} & 0\\
0 & \zeta X_{\zeta}
\end{array}\right)\left(\begin{array}{cc}
0 & I\\
Y & 0
\end{array}\right)=\left(\begin{array}{cc}
0 & X_{\zeta}\\
\zeta X_{\zeta}Y & 0
\end{array}\right)\\
VU & = & \left(\begin{array}{cc}
0 & I\\
Y & 0
\end{array}\right)\left(\begin{array}{cc}
X_{\zeta} & 0\\
0 & \zeta X_{\zeta}
\end{array}\right)=\left(\begin{array}{cc}
0 & \zeta X_{\zeta}\\
YX_{\zeta} & 0
\end{array}\right)=\left(\begin{array}{cc}
0 & \zeta X_{\zeta}\\
\zeta^{2}X_{\zeta}Y & 0
\end{array}\right)=\zeta UV.
\end{eqnarray*}
Define a $\ZZ_{2n}\times\ZZ_{2n}$-grading on $M_{2n,n}(E)$ by $M_{2n,n}(E)_{(k,l)}=U^{k}V^{l}\otimes E_{\left(-1\right)^{l}}$,
where $E=E_{1}\oplus E_{-1}$ is the usual grading on the Grassmann
algebra. This induces a minimal regular grading on $M_{2n,n}(E)$
with commutation function $\theta$ determined by 
\[
\theta\left[(1,0),(1,0)\right]=1\quad\theta\left[(1,0),(0,1)\right]=\zeta^{-1}\quad\theta\left[(0,1),(0,1)\right]=-1.
\]
\\

\end{example}
We consider the bicharacters which correspond to (some special cases
of) the gradings just described.
\begin{enumerate}
\item $\left(\left\{ \pm1\right\} ,\tau\right)$: where $\tau(1,1)=\tau(1,-1)=\tau(-1,1)=1,\;\tau(-1,-1)=-1$.
\item $(\ZZ_{p^{m}}\times\ZZ_{p^{m}},\eta_{p^{m}})$: where $a=(1,0),\; b=(0,1)$
and $\eta_{p^{m}}(a,a)=\eta_{p^{m}}(b,b)=1,\;\eta_{p^{m}}(a,b)=\zeta$
for some primitive $p^{m}$ root of unity $\zeta$.
\item $\left(\ZZ_{2^{m}}\times\ZZ_{2^{m}},\epsilon_{2^{m}}\right)$: where
$a=(1,0),\; b=(0,1)$ and $\epsilon_{2^{m}}(a,a)=1,\;\epsilon_{2^{m}}(b,b)=-1,\;\epsilon_{2^{m}}(a,b)=\zeta$
for some primitive $2^{m}$ root of unity $\zeta$. \end{enumerate}
\begin{defn}
A bicharacter is called\emph{ basic} if it is isomorphic to one of
bicharacters (1)-(3).\end{defn}
\begin{rem}
\label{unique-type-bicharacter}Let $a$ and $b$ be the elements
which appear in the definition of the second or third basic bicharacter
(of order $p^{m}$ and with primitive root of unity $\zeta$). Note
that for any prime to $p$ integer $k$, we have $G=\left\langle ka\right\rangle \times\left\langle b\right\rangle $,
$\theta(ka,ka)=1$ and $\theta(ka,b)=\zeta^{k}$, so for different
choices of primitive $p^{m}$-roots of unity we get isomorphic bicharacters.
In particular, for each group $G$, if $\theta_{1},\theta_{2}$ are
two basic bicharacters of the same type on $G$, then they are isomorphic.
\end{rem}
The next 3 results were proved in \cite{alekseevich_commutation_1996}
(see Lemma 6, Lemma 7 and Theorem 1).
\begin{lem}
The basic characters $\tau,\eta_{p^{m}},\epsilon_{2^{n}}$, where
$m,n\geq1$, are nonisomorphic. Furthermore, the set of nondegenerate
irreducible bicharacters on non trivial groups coincides with the
set $\{\tau,\eta_{p^{m}},\epsilon_{2^{n}}:m\geq1,n\geq2\}$ (that
is the set of basic bicharacters except $\epsilon_{2}$).
\end{lem}
We can now write each nondegenerate skew-symmetric bicharacter $\theta$
as a product of basic bicharacters. In general, this presentation
is not unique. Nevertheless, using the isomorphisms below, there exists
a canonical presentation for any $\theta$.
\begin{enumerate}
\item $\epsilon_{2^{n}}\otimes\epsilon_{2^{m}}\cong\epsilon_{2^{n}}\otimes\eta_{2^{m}}$
for all $1\leq n\leq m\;\in\NN$.
\item $\epsilon_{2^{n}}\otimes\tau\cong\eta_{2^{n}}\otimes\tau$ for all
$n\in\NN$.
\item $\tau\otimes\tau\cong\epsilon_{2}$.\end{enumerate}
\begin{thm}
Let $G$ be a finite abelian $p$-group and $\theta$ a nondegenerate
skew-symmetric bicharacter on $G$. Then there is a unique canonical
presentation $G=\prod_{1}^{n}G_{i}$ , $\theta\cong\bigotimes\theta\mid_{G_{i}}$,
such that for each $i$, the bicharacter $\left(G_{i},\theta\mid_{G_{i}}\right)$
is a basic bicharacter, where at most one basic bicharacter is of
type 1 or type 3.
\end{thm}
Note, in particular, that the three types of bicharacters in the last
theorem correspond the three types of algebras in \corref{3_types_for_Z_2_simple}.
We can now determine the abelian groups $G$ which admit a nondegenerate
regular grading or equivalently a nondegenerate skew-symmetric bicharacter.
To this end, let $\theta$ be a nondegenerate skew-symmetric bicharacter
on an abelian group $G$. If the canonical decomposition of $\theta$
has no factor isomorphic to $\tau$, then the group $G$ is isomorphic
to $N\times N$ (i.e. central type, abelian). On the other hand, if
one of the components in the canonical decomposition is isomorphic
to $\tau$, then we have $G=H\times\nicefrac{\ZZ}{2\ZZ}$, where
\begin{enumerate}
\item $H$ is a group of central type determined by $H=\left\{ h\in G\;\mid\;\theta(h,h)=1\right\} $.
\item $\theta\mid_{\nicefrac{\ZZ}{2\ZZ}}\cong\tau$.
\end{enumerate}
In the general case, if $H$ is of central type, then clearly $H\times\nicefrac{\ZZ}{2\ZZ}$
admits a nondegenerate commutation function. However, the following
example shows that one may have nondegenerate commutation functions
on groups which are not of this kind. For instance, in \exref{dihedral},
we considered the 2-cocycle $\alpha\in Z^{2}(D_{8},\FF^{\times})$
induced by the extension 
\[
\xymatrix{1\ar[r] & \left\{ \pm1\right\} \ar[r] & Q_{16}\ar[r] & D_{8}\ar[r] & 1}
,
\]
where we view the group $\left\{ \pm1\right\} $ as a subgroup of
$\FF^{\times}$. One can check easily that the (natural) $D_{8}$-grading
on the Grassmann envelope $A=E\dotimes\FF^{\alpha}D_{8}$ is regular
and nondegenerate.

\bibliographystyle{my-style}
\bibliography{regular}

\end{document}